\documentclass{amsart}

\usepackage{amsmath}
\usepackage{amsthm}
\usepackage{amssymb}
\usepackage{braket}
\usepackage{mathtools}

\usepackage[all,cmtip,2cell]{xy}
\usepackage{soul}

\usepackage{tikz}
\usetikzlibrary{arrows,decorations.pathmorphing,backgrounds,positioning,fit,petri}

\usepackage[colorlinks=true, allcolors=blue]{hyperref}

\usepackage{color}

\newcommand{\alertb}[1]{{\color{blue}#1}}

\newtheorem{thm}{Theorem}[section]
\newtheorem{theorem}[thm]{Theorem}
\newtheorem{corollary}[thm]{Corollary}
\newtheorem{lemma}[thm]{Lemma} 
\newtheorem{proposition}[thm]{Proposition}

\theoremstyle{definition}
\newtheorem{definition}[thm]{Definition}
\newtheorem{example}[thm]{Example}
\newtheorem{observation}[thm]{Observation}
\newtheorem{remark}[thm]{Remark}
\newtheorem{notation}[thm]{Notation}

\newcommand{\Sets}{\mathrm{Sets}}
\newcommand{\sSets}{\mathrm{sSets}}
\newcommand{\ssSets}{\mathrm{ssSets}}
\newcommand{\diGraphs}{\mathrm{diGraphs}}
\newcommand{\diGraphss}{\mathrm{diGraphs}_{\leq 1}}

\newcommand{\Graphs}{\mathrm{Graphs}}
\newcommand{\stGraphs}{\mathrm{Graphs}_{\leq 1}}

\newcommand{\Vect}{\mathrm{Vect}}
\newcommand{\Vectk}{\mathrm{Vect}_\mathbb{K}}

\newcommand{\Hom}{\mathrm{Hom}}
\newcommand{\Fun}{\mathrm{Fun}}

\newcommand{\pr}{\mathrm{pr}}

\newcommand{\Cov}{\mathrm{Cov}}

\newcommand{\open}{\mathrm{Open}}

\newcommand{\Et}{\mathrm{Et}}
\newcommand{\gross}{\mathrm{gross}}

\newcommand{\oomega}{\omega}

\newcommand{\id}{\mathrm{id}}

\newcommand{\cA}{\mathcal{A}}
\newcommand{\cC}{\mathcal{C}}

\newcommand{\cO}{\mathcal{O}}
\newcommand{\cT}{\mathcal{T}}
\newcommand{\cF}{\mathcal{F}}
\newcommand{\cG}{\mathcal{G}}
\newcommand{\cU}{\mathcal{U}}
\newcommand{\cB}{\mathcal{B}}

\newcommand{\cV}{\mathcal{V}}
\newcommand{\cD}{\mathcal{D}}

\newcommand{\cP}{\mathcal{P}}
\newcommand{\cR}{\mathcal{R}}

\newcommand{\Pre}{\textnormal{Pre}}
\newcommand{\Sh}{\textnormal{Sh}}

\newcommand{\coSh}{\textnormal{coSh}}

\newcommand{\extd}{\mathrm{d}}

\newcommand{\R}{\mathbb{R}}

\newcommand{\N}{\mathbb{N}}
\newcommand{\K}{\mathbb{K}}
\newcommand{\bk}{\mathbb{K}}
\newcommand{\ccF}{{F}}

\newcommand{\al}{\alpha}

\newcommand{\be}{\beta}
\newcommand{\de}{\delta}

\newcommand{\lra}{\longrightarrow}

\begin{document}

\title{Sheaves on Graphs and their Differential Calculi}
\maketitle

\centerline{R. Fioresi${}^{\star}$, A. Simonetti${}^\star$, F. Zanchetta${}^\star$}

\bigskip

\centerline{$^\star${\sl FaBiT, Universit\`{a} di
Bologna}}

\centerline{\sl Via Piero Gobetti 87, 40129 Bologna, Italy}

\begin{abstract} 
In this paper we explore the link between the theory of sheaves on graphs and noncommutative geometry showing that many concepts and constructions in the latter can be generalized and enhanced using methods coming from the former. They include notions such as Laplacians and connections, important in the theory of discrete noncommutative geometry, that are here explored with sheaf theoretic methods and using the language of (semi)simplicial sets.
\end{abstract}

\section{Introduction}
In recent years graph theory has seen renewed interest, due to advances in machine learning \cite{TDLpos, DLbook, br1}, where data appear with geometric or topological
structure. Graph Neural Networks (GNNs) leverage graph Laplacians in their very design \cite{KW16, Sage17, bronstein}, but were soon
generalized to directed graphs, whose Laplacians are studied in \cite{Bauer12}, and multigraphs.
This motivated formalizations via semisimplicial sets \cite{Spivak, bodnar2}, leading also to simplicial neural networks,
as well as other novel architectures modelling more complex
structures on graphs, such as sheaves \cite{HG20, HG19}, leading to 
the notion of Sheaf Neural Networks (SNNs) \cite{bodnar}.

\par SNNs, as originally formulated,
relied on the theory of cellular sheaves, that is the theory of sheaves on posets coming
from regular cell complexes \cite{HG19, curry}. For a given cell complex, a cellular sheaf amounts
to the datum of a vector space or an Hilbert space
for each simplex, i.e. local data, together with compatibility relations given by linear transformations
between cells of incident ascending dimension.
As in classical homotopy theory, the datum of a cellular sheaf can be arranged
in a cochain complex graded by the dimension of cells: if the cellular sheaf takes values in the category of Hilbert spaces, adjoints (the dagger category structure of Hilbert spaces) allow us to define Laplacians that specialize to the ordinary graph Laplacian and the Hodge Laplacians in the case of constant cellular sheaves.
{Going back to the case of a (multiedge, directed) graph $G$, we have a poset $P_G$ associated to it by looking at the incidence relations.
In this case, a cellular sheaf is the datum of a sheaf on the topological space $A(P_G)$ endowed
with the Alexandrov topology, where we
note that each node corresponds to an irreducible open set and
edges correspond to the intersections of these sets.
This amounts to the datum of vector spaces $F(v), F(e)$ for each vertex $v$
and edge $e$
and linear maps $F(v)\to F(e)$ if $v$ is an endpoint of $e$. If one considers linear maps in the opposite
directions one gets a cosheaf, that can be also interpreted as a sheaf on the dual of the finite
topological space $A(P_G)$.

In the language of Grothendieck topologies,
cosheaves on a graph are exactly the sheaves on the Grothendieck site obtained by considering the
category of subgraphs of a given graph (morphisms being graph inclusions), and the Grothendieck topology having as coverings jointly surjective families of subgraphs. This point of view, pursued by Friedman in \cite{FrTop} and subsequent works
led to his famous proof of the Hanna-Neumann conjecture via topos theoretic methods in graph theory.
The fact that graphs can provide a convenient way to organize information related with more complex objects is surfacing also in many other fields. The language of noncommutative geometry \cite{bm} appears spontaneously as the right framework to cast questions on the differential geometry of discrete structures as 
graphs, whose noncommutative nature is manifest from the bimodule structure on
their first order differential calculi \cite{majidpaper, dimakis}.
In this context we can model the discrete counterparts of classical notions such as curvature,
connections, parallel transport, etc.
and some of these discrete notions have been extended to the context of ``bundle-valued-forms"
(see \cite{gaybalmaz}) and appear surprisingly
in the context of Sheaf Neural Networks \cite{bodnar}.
Different fields in pure and applied mathematics are then coming naturally together:
the present paper represents an effort towards their seamless meshing, which shall
bring some deeper understanding on their connections.

One of the main difficulties that prevents the full transfer of the ideas and methods developed
in fields such as differential geometry or algebraic geometry to discrete geometry is the fact that
many discrete geometric objects are often built in analogy with their non-discrete counterparts and not
\emph{from} their non-discrete counterparts (or viceversa) in a more conceptual way.

Surprisingly, some recent research in the field of algebraic geometry has started to address exactly this issue. In some recent work, Salas and their collaborators (see \cite{salas} and the references therein) study in more detail the notion of finite ringed spaces, i.e. ringed spaces having as underlying topological space a finite topological space. In our recent work \cite{FSZ26} we take a novel point of view on this theme and we show that highly structured geometric objects such as schemes (and differentiable manifold) can be seen as a finite 2 dimensional semisimplicial set together with a sheaf of rings on its face-incidence poset satisfying certain assumptions: the analogy with cellular sheaves becomes then striking. The approach of \cite{FSZ26}, clearly reminiscent of the Grothendieck theory of descent \cite{Vis}, allows us to see how to a ``discrete" object (a 2-dimensional semisimplicial set) one can attach extra data (e.g. a sheaf of rings on the poset of simplices of a 2-dimensional semisimplicial set) to recover a ``non-discrete" object (e.g. a scheme or a manifold). Notice that no higher dimensional simplices are involved in this process to reconstruct $X$. \par
The main purpose of this paper, to be thought as the ``discrete companion" of \cite{FSZ26} is then to
provide mathematical foundations to the theory of sheaves on graphs, and their differential operators
in order to connect all the mathematical subfields where similar ideas and tools are used
for different purposes, including very concrete
applications as mentioned above (see also \cite{NG25} for more related applications of algebraic geometry).
This aim requires not only to recast known results in a different way, but also to develop new mathematics together with a novel point of view connecting
established theories, as quantum differential calculus and discrete differential geometry.
It is worth noticing that Ladkani in \cite{SP06} studied the categories of sheaves over finite posets, while Liu in \cite{SQ13} studied stacks and torsors over quivers. Their work, however, is only partially related with our.

\medskip
We shall now describe the contents of the paper, highlighting our main results. \par
In Sec. \ref{graph-sec} we review standard material on semisimplicial sets
to introduce graphs and we establish our notation, see \cite{GJ1999, SP09, EW18} as references. 
\par
In Sec. \ref{gr-sec}, we discuss Grothendieck topologies (pretopologies if one adopts the terminology of \cite{EGA4}), tailoring the discussion to the ``discrete" cases, that is, those cases where every object of a site can be covered by so called ``gross objects", which are the analogous of irreducible open subsets,
that we introduce and study. This is related with, but different from the theory developed in \cite{FrTop}.
We also introduce and study the notion of Laplacian pairs to generalize cellular sheaf Laplacians. Laplacian pairs basically consists of the datum of a sheaf, a cosheaf and a covering for a Grothendieck topology: the \v{C}ech complexes arising from the (co)sheaves involved allow to define sheaf Laplacians in situations more general than the one we get considering a cellular sheaf and its adjoint.
\par
In Sec. \ref{sec:shgraph} we study some simple Grothendieck topologies on graphs and semisimplicial sets, opening the doors to the introduction of more general ones. We link the general theory to the one of cellular sheaves and we provide a unified treatment of all the different topologies that have arisen so far in this context.
We also study the basic Laplacians arising in this context.
\par
In Sec. \ref{lapl-sec} we extend the discrete differential geometry on graphs developed in \cite{dimakis, majidpaper} to the case of directed graphs admitting more than one edge connecting each pair of vertices. In addition, we link the notions of discrete connection and curvature found in discrete differential geometry
in the quantum  language \cite{bm} to the notion of discrete parallel transport as in \cite{gaybalmaz}.
Finally, we show how all the ordinary Laplacians that can be defined starting from a sheaf on a graph, a parallel transport on a graph and a connection on a graph agree (up to a constant term): this is summarized by the following statement (see Section \ref{lapl-sec} for an explanation of the notation and the terminology).

\begin{theorem}[\ref{theo:sheaf_lap}]
Let $G\in \diGraphs_{\leq 1}$ be a bidirected graph, $\cF$ a vector bundle of rank $n$ on $G$, $M$ the free $A_G$-bimodule associated to it. Assume to have isomorphisms $M \cong M^*$ and $\Gamma^1\cong (\Gamma^1)^*$ as in Def. \ref{def:bochnerlapl} and consider $(,)$ the generalized quantum metric determined by scalars $\{\lambda_{x\to y\to x}\}$ all equal to $1$. Then:

\begin{enumerate}
\item If $\mathcal{R}^\cF$ is a parallel transport, then $_\theta\Delta^{M} = -\Delta_{\cF}$.
\item If $\cF$ is a sheaf of inner product spaces and $\cF_{v\leq e}^\ast = \cF_{v\leq e}^{-1}$
i.e. $\cF$ in an $\mathrm{O}(n)$-bundle,
then $\Delta^{\mathrm{B}} = L_{\cF}$.
\end{enumerate}
\end{theorem}

\medskip
{\bf Notation.} We shall record some of the notations we employ throughout the paper here for the convenience of the reader.
\begin{itemize}
    \item We denote as ($\Delta_+$) $\Delta$
the category having as objects the partially
ordered sets $[n]=\lbrace 0\rightarrow\cdots\rightarrow n\rbrace$, $n\in\N$,
and arrows the (injective)
order preserving maps between them. For any $n\in\N$ e denote as
($\Delta_{n,+}$) $\Delta_{n}$
its full subcategory of objects having cardinality smaller or equal than $n$.
\item We will denote with $\Sets$ and $\Vectk$ the categories of sets and of finitely generated vector spaces over a base field $\mathbb{K}$ respectively.
\item Given a set $X$ a preorder on $X$ is a binary relation that is reflexive and transitive. If this relation is also antisymmetric, then we have a partial order: in this case we will say that $X$ is a poset.
\item Given a category $\cC$ we will denote as $\cC^{op}$ its dual (or opposite) category. Finally, given two categories $\cA$ and $\cB$ we shall denote as $\Fun(\cA,\cB)$ or as $\cB^{\cA}$ the category of the covariant functors $\cA\rightarrow\cB$ and as $\Pre(\cA,\cB):=\Fun(\cA^{op},\cB)$ the category of contravariant functors $\cA\rightarrow\cB$. We refer to the objects of $\Pre(\mathcal{A},\mathcal{B})$ as
\textit{presheaves}.
If the category $\cB$ is the category $\Sets$ or it is clear from the context,
we shall denote $\Pre(\cC,\cB)$ simply as $\Pre(\cC)$.
\end{itemize}

\section{Graphs as semisimplicial sets}\label{graph-sec}
In this section, we first
recall what a (semi)simplicial set is and then we give the definition of directed graph in these terms,
together. We also recollect some results we will need in the sequel \cite{SP09}.

\subsection{The categories of
(semi)simplicial sets and digraphs}\label{ssets-sec}

\begin{definition}\label{ssets-def}
We define the category of \textit{(semi)simplicial} sets to be the category
$\sSets:=\Pre(\Delta,\Sets)$ ($\ssSets:=\Pre(\Delta_+,\Sets)$).
For a given (semi)simplicial set $X$, for every $n\in\N$,
we shall denote the set $X([n])$ as $X_n$ and call its elements \textit{simplices}.
Replacing the category of sets with an arbitrary category $\cC$ gives us
the notion of a \textit{(semi)simplicial object in} $\cC$.
We define \textit{(semi)cosimplicial sets}
as $\Fun(\Delta,\Sets)$ ($\Fun(\Delta_+,\Sets)$)
and $n$-dimensional (semi)simplicial sets and
(semi)cosimplicial sets as $\Pre(\Delta_n,\Sets)$, ($\Pre(\Delta_{n,+},\Sets)$),
$\Fun(\Delta_n,\Sets)$, ($\Fun(\Delta_{n,+},$ $\Sets)$) respectively.
We say that a (co)semisimplicial set $X$ has dimension $n$
if $X_n\neq\emptyset$ and $X_m=\emptyset$ for $m>n$. 
\end{definition}

\begin{remark}
By definition, our $n$-dimensional simplicial sets are not simplicial sets,
while in literature a simplicial set is usually said to have dimension $n$ if it is
isomorphic to its $n$-skeleton \cite{GJ1999}.
However, the inclusion functor $\Delta_n\subseteq\Delta$ induces a truncation
function $tr_n:\sSets\rightarrow\sSets_n$ having as a fully faithful left adjoint
$sk_n:\sSets_n\rightarrow\sSets$ the $n$-skeleton functor. For a given simplicial set
$X$, $sk_n\circ tr_n (X)$ is the usual $n$-skeleton (see also \cite[\href{https://kerodon.net/tag/04ZY}{Tag 04ZY}]{kerodon}).
This recovers the link between our definition and the most common definition.
\end{remark}

\begin{definition}
We define the category of \textit{directed graphs} to be $\Pre(\Delta_{1,+},\Sets)$ and
we denote it $\diGraphs$.
Replacing the category of sets with an arbitrary category $\cC$ gives us
the notion of \textit{directed graphs in} $\cC$.
\end{definition}

Given a digraph $G$, we will call the sets $G_0$ and $G_1$ the set of \emph{vertices} and \emph{edges} respectively. The maps $d_0,d_1:[0]\hookrightarrow [1]$ induce via $G$ two maps $d_0^\ast,d_1^\ast:G_1\rightrightarrows G_0$ called \emph{head} and \emph{tail} respectively.
Unpacking the definitions, a digraph $G$ is simply the datum
$(V_G,E_G,h_G,t_G)$ of two sets $V_G=G_0$ (\textit{vertices}), $E_G=G_1$ (\textit{edges}) and two functions
$h_G=d_0^\ast$, $t_G=d_1^\ast:E_G\rightarrow V_G$.
We denote elements in $E_G$ as $e_{v_1 \lra v_2}$, or 
{$v_1 \lra v_2$}
where $v_2=h_G(v_1 \lra v_2)$
and $v_1=t_G(v_1 \lra v_2)$.
We may represent such information in a diagram in $\Sets$ of the following form
\begin{displaymath}
\xymatrix{E_G\ar@<1ex>[r]^{h_G} 
\ar@<-1ex>[r]_{t_G}& V_G
}
\end{displaymath}
Note that this definition of directed graph allows graphs with multiple edges with same head and tail (sometimes called \emph{directed multi-graphs}) and self-loops. The notion of morphism between digraphs can be unpacked similarly.
By definition of semisimplicial set, for a given $G \in \diGraphs$ both $V_G$ and $E_G$ can be infinite sets.

\begin{remark}\label{dim1simp}
Semisimplicial sets and simplicial sets are closely related.
The latter category is the most frequently used in the context of simplicial homotopy
theory because of the nice properties of the geometric realisation
\cite{May1992, GJ1999}.
The former appeared more frequently in early modern algebraic geometry, for example in \cite{SGA4}. For the purpose of this work, semisimplicial objects provide a more natural context because of their easier combinatorics. The reader must be assured, however, that they are not an artificial category: there exists a geometric realisation for semisimplicial objects (the
``fat realisation'') and the geometric realisation of a simplicial set $X$ is always homotopy equivalent to the geometric realisation of the semisimplicial set obtained from $X$ precomposing with the inclusion $\Delta_+^{op}\subseteq \Delta^{op}$ (see \cite[Section 2]{EW18}
and the discussion therein).
Moreover in dimension lower or equal to one the categories
of simplicial sets and semisimplicial sets are equivalent, 
\cite[\href{https://kerodon.net/tag/001N}{Proposition 001N}]{kerodon},
thus directed graphs may be viewed as simplicial sets.
\end{remark}

As digraphs are simply presheaves of sets, we have the following
\cite{KS2006}.

\begin{proposition}
The category of digraphs has all limits and colimits and they are computed pointwise.
\end{proposition}

In particular, fibre products of digraphs exist. 
Notice that for a given digraph $G$, because of the universal property of products there exists a unique map $i$ such that the following diagram commutes 

\begin{equation}\label{diagr-i}
\xymatrix{ & E_G\ar[dr]^{t_G}\ar[dl]_{h_G}\ar[d]^i & \\
V_G & V_G\times V_G\ar[r]_{\pr_2}\ar[l]^{\pr_1} & V_G
}
\end{equation}
Conversely, any map $i:E_G\rightarrow V_G\times V_G$ determines uniquely a pair $(h_G, t_G)$ such that the previous diagram commutes. When $i$ is injective, this point of view
recovers the more elementary description of the set of edges of a directed digraph as a
subset of $V_G\times V_G$, where we assume 
no more than one edge between any two vertices in each direction.
We denote with $\diGraphs_{\leq1}$ the full subcategory of $\diGraphs$ whose objects
have $i$ as in (\ref{diagr-i}) injective.

We end this treatment of (semi)simplicial sets by defining the category of \textsl{(semi) \break simplices} of
$X$, $\Gamma(X)$ ($\Gamma_+(X)$), for a given (semi)simplicial set $X$, see \cite{hovey}.

By Yoneda's Lemma,
the maps $\Delta^n\rightarrow X$ (resp.  $\Delta_+^n\rightarrow X$)
are in bijection with the elements of $X([n])=X_n$,
where $\Delta^n=\Hom_\Delta(-,[n])$ (resp. $\Delta^n_+=\Hom_{\Delta_+}(-,[n])$).
So we take such maps as the objects for $\Gamma(X)$ ($\Gamma_+(X)$)
as $n$ varies. {In particular, for $G \in \diGraphs$ the objects of this category will be in bijection with all vertices $G_0$ and all edges $G_1$.}
To define the morphisms of $\Gamma(X)$, ($\Gamma_+(X)$), we want to specify
when we have an arrow between two objects 
in $\Gamma(X)$ that is, $a \in X_n$ and $b \in X_m$.
We say we have such an arrow, whenever we can
write a commutative diagram:
\begin{equation}\label{arr-sim-eq}
\xymatrix{\Delta^n \ar[rr] \ar[dr]_a & & \Delta^m\ar[dl]^b\\
 & X &}
\qquad \hbox{resp.} \quad
\xymatrix{\Delta^n_+ \ar[rr] \ar[dr]_a & & \Delta_+^m\ar[dl]^b\\
 & X &}
\end{equation}
where again we employ Yoneda's lemma.

\begin{definition}\label{catofsimp}
For a given (semi)simplicial set $X$, we define
$\Gamma(X)$ ($\Gamma_+(X)$)
the \textit{category of (semi)simplices of} $X$ as the category with objects
the maps $\Delta^n\rightarrow X$ (resp.  $\Delta_+^n\rightarrow X$)
and morphisms the diagrams in (\ref{arr-sim-eq}). Notice that a morphism $f:X\rightarrow Y$ of
(semi)simplicial sets induces a functor
between their respective categories of simplices that we will denote as $\Gamma(f)$
($\Gamma_+(f)$ respectively).
\end{definition}
\begin{observation}\label{digraph-poset}
For a semisimplicial set $X$, the objects and the arrows of $\Gamma_+(X)$
can be used to define a poset $P_X$ where the
elements are the simplices of $X$ and $x\leq y$
if and only if $\Hom_{\Gamma_+(X)}(x,y)$ is non empty.
\end{observation}
\begin{remark}\label{cat-poset}
Notice that for a semisimplicial set $X$ there is a non-canonical isomorphism between the category $\Gamma_+(X)$ and the category associated with the poset $P_X$
{having as objects the elements of $P_X$ and arrows $x \leq y$.}
\end{remark}

\subsection{Undirected Graphs}\label{sec:undgr}

Recall that an undirected graph $G$ is a pair $(V,E)$ of vertices and edges,
where the edges are defined as a multiset of
unordered pairs of vertices and morphisms are defined accordingly
(see \cite{diestel, godsil} for the case $V$ and $E$ finite sets, though here
we are more general).
Let $\Graphs$ denote their category. We call the category of
simple undirected graphs, i.e. undirected graphs having at most
one edge connecting each pair of vertices, as $\stGraphs$.

We define the functor $\Theta:\Graphs\rightarrow\diGraphs$ as follows.
For $G=(V,E)$ in $\Graphs$:
\begin{itemize}
\item $\Theta(G)_0 = V$ 
\item for every edge $e$ in $E$ with {distinct}
endpoints $v_1,v_2$, we have two edges $v_1\rightarrow v_2, v_2\rightarrow v_1$
in $\Theta(G)_1$, with $h_{\Theta(G)}(v_1\rightarrow v_2)=t_{\Theta(G)}(v_2\rightarrow v_1)$ and
$h_{\Theta(G)}(v_2\rightarrow v_1)=t_{\Theta(G)}(v_1\rightarrow v_2)$. 
For every loop $e$ in $E$ at the vertex $v$ we
have one edge $v \lra v$ in $\Theta(G)_1$.
\end{itemize}

As for morphisms, if $f\in \Hom(G,H)$, then $\Theta(f)$ on $\Theta(G)_0$ coincides with $f$. Moreover, if $e_{v_1\rightarrow v_2}, e_{v_2\rightarrow v_1}$ are two elements in $\Theta(G)_1$
as above, then $\Theta(f)(e_{v_1\rightarrow v_2})=f(e)_{f(v_1)\rightarrow f(v_2)}$ and $\Theta(f)(e_{v_2\rightarrow v_1})=f(e)_{f(v_2)\rightarrow f(v_1)}$. Since all elements of $\Theta(G)_1$ are of this form, this concludes the definition of $\Theta$. It is a standard check that $\Theta$ is indeed a functor. Besides we have the following.

\begin{proposition}\label{prop-graphffemb}
The functor $\Theta: \Graphs\rightarrow\diGraphs$ is faithful. If restricted
to the category of standard graphs with at most one edge for each pair of vertices,
$\stGraphs$, then $\Theta:\stGraphs\rightarrow \diGraphs$ is full as well.     
\end{proposition}
\begin{proof}
The fact that $\Theta_{G,H}:\Hom (G,H)\longrightarrow\Hom (\Theta(G), \Theta(H))$ is injective follows directly from the construction of $\Theta$.
To prove that $\Theta$ is full when restricted to $\stGraphs$,
we notice that morphisms in $\stGraphs$ and between graphs in the image of $\Theta_{|\stGraphs}$ are uniquely determined by their behavior on the vertices.
\end{proof}

Finally, note that there is another way of viewing an undirected graph as a directed one,
that is choosing an orientation on it.

\begin{observation}\label{graph-poset} In Observation \ref{digraph-poset} we have given a way to associate a poset to a semisimplicial set, and therefore to a directed graph.
We can do the same with undirected graphs as follows.
Let $G=(V,E)\in\Graphs$ be an undirected graph. Define the poset $\mathcal{P}_G$ associated to it having as underlying set $V\cup E$ and where $x\leq y$ if and only if $x$ is a vertex of the edge $y$ or $x=y$. This poset is compatible with the one defined in Observation \ref{digraph-poset} in the sense that $\mathcal{P}_G=P_{\Theta(G)}$.
\end{observation}

}

\section{Grothendieck Topologies and Sheaves}
\label{gr-sec}
In this section we take advantage of the language of Grothendieck topologies
we need later on for
the study of sheaves on graphs, \cite[Exposé ii]{SGA4}, \cite{Vis}, \cite{SP}.

\subsection{Grothendieck topologies, gross coverings, bisheaves}

Given a category $\cC$ we denote simply as $\Pre(\cC)$ the category of presheaves
of sets on $\cC$.
Given a Grothendieck site $(\cC,\cT)$ we will denote as 
$\Sh(\cC,\cT)$ the category of sheaves of sets on it. In the case of sheaves taking values in some category $\cA$, we shall use the notation $\Sh((\cC,\cT),\cA)$ if confusion may arise.

\begin{definition}
Let be $\cC$ a category, and $\cU=\lbrace \varphi_i:U_i\rightarrow U\rbrace_{i\in I}$ be
a covering of $U$, that is a family of arrows
with fixed target $U$.
We say that $\cU$ is \textit{gross} if there exists $j\in I$ such that $\varphi_j$
is an isomorphism.
Given a Grothendieck topology $\cT$ on $\cC$, an object $U \in \cC$ is called \textit{gross} if every covering of $U$ is gross.

If $\mathcal{V}$ $=\lbrace V_j\rightarrow U\rbrace_{j\in J}$
is a covering of $U$, we say that
$\mathcal{V}$ is a \emph{refinement} of $\mathcal{U}$ 
if for every $j\in J$ there exists $i\in I$ such that 
$V_j\rightarrow U$ factors through $U_i\rightarrow U$.
\end{definition}

The definition of gross object is adapted from the one that can be found in \cite[Section 2.5]{FrTop}, 
(notice that in \cite{FrTop} sieves are used).
To help the geometric intuition, we recall (see \cite[Section 2.5]{FrTop}) that an open set $U$ in a topological space
is gross 
if and only if it is irreducible, i.e. if and only if it is
not the union of its proper open subsets.
So $U$ is gross if and only if
there is a point $p\in U$ such that any open set containing $p$ contains $U$.

\begin{lemma}\label{lemmagroseq}
For a given category $\cC$, let us consider $\cU=\lbrace f_i:U_i\rightarrow U\rbrace_{i\in I}$ and
$\cV=\lbrace g_j:V_j\rightarrow U\rbrace_{j\in J}$ two gross families of arrows with fixed target. Then $\cU$ is
equivalent to $\cV$,
i.e. $\cU$ is a refinement of $\cV$ and vice-versa.
\end{lemma}
\begin{proof}
We show that $\cV$ is a refinement of $\cU$: the converse can be proven analogously. As $\cU$ is gross,
there exists $\alpha\in I$ such that $f_\alpha:U_\alpha\rightarrow U$ is an isomorphism. If we define for
all $j\in J$ $\varphi_j:=f_\alpha^{-1}\circ g_j$ we have that for all $j\in J$ the map $g_j$ factors through
$f_\alpha$ as $f_\alpha\circ\varphi_j =g_j$.
\end{proof}

\begin{remark}
Lemma \ref{lemmagroseq}
shows that any family of arrows with fixed target $U\in\cC$,
for $\cC$ given, is a refinement of any gross family of arrows with fixed target $U$.
\end{remark}

\begin{observation}
Let be $(\cC,\cT)$ a Grothendieck site, where every object of $\cC$ is gross.
Then, by Lemma \ref{lemmagroseq}, $\cT$ is equivalent to the chaotic
topology
(see for example \cite[Section 2.3.5]{Vis} for the notion of equivalence between Grothendieck topologies), 
i.e. the topology where all coverings are isomorphisms. 
\end{observation}

\begin{lemma}
Let $(\cC,\cT)$ be a Grothendieck site and $U\in\cC$ a gross object. If
$f:U\rightarrow V$
is an isomorphism, then $V$ is a gross object.
\end{lemma}
\begin{proof}
If $V$ is not gross, then there exists $\cV=\lbrace g_j:V_j\rightarrow V\rbrace_{j\in J}\in\Cov(V)$ that is not a
gross cover. Then, its base change along $f$ is
a non gross cover of $U$.
\end{proof}

When considering sites, we have the notion of continuous functor (see
  \cite[\href{https://stacks.math.columbia.edu/tag/00WV}{Definition 00WV}]{SP}).
\begin{definition} 
Let $(\cC,\cT)$ and $(\cD,\cG)$ be Grothendieck sites. A functor $F:\cC\rightarrow \cD$ is called \textit{continuous},
if for every $U\in\cC$, $\mathcal{U}=\lbrace f_i:U_i\rightarrow U\rbrace_{i\in I}\in\Cov(U)$:
\begin{itemize}
\item $\lbrace F(f_i):F(U_i)\rightarrow F(U)\rbrace_{i\in I}\in\Cov(F(U))$,
\item for any morphism $Y\rightarrow U$ in $\cC$, the morphism $F(Y\times_UU_i)\rightarrow F(Y)\times_{F(U)}F(U_i)$ is an isomorphism.
\end{itemize}
Given a continuous functor $F:(\cC,\cT)\rightarrow (\cD,\cG)$, there is a well defined functor $F^p :\Sh(\cD,\cG)\rightarrow \Sh(\cC,\cT)$, $P\mapsto P\circ F$ (see \cite[\href{https://stacks.math.columbia.edu/tag/00WW}{Lemma 00WW}]{SP}).
\end{definition}

Let $(\cC,\cT)$ be a Grothendieck site. We denote by $\cC_{\mathrm{gross}}$ the full subcategory
of $\cC$ consisting of gross objects, writing $\cC_{\mathrm{gross}}(\cT)$ if confusion may arise.

\begin{theorem}\label{GrossBasisThm}
Let $(\cC,\cT)$ be a Grothendieck site.
Assume that for every object $U\in\cC$ there exists a covering $\lbrace U_i\rightarrow U\rbrace_{i\in I}$ of it
such that for every $i\in I$, $U_i$ is a gross object. Then the inclusion functor 
 $\cC_{\mathrm{gross}}\rightarrow \cC$
is continuous and induces an equivalence $\Pre(\cC_{\mathrm{gross}})\simeq\Sh(\cC,\cT)$.
\end{theorem}
\begin{proof}
It suffices to check that the assumptions of
\cite[\href{https://stacks.math.columbia.edu/tag/03A0}{Tag 03A0}]{SP} are satisfied.
\end{proof}

This theorem is key to our treatment, as it will allows us,
to define a sheaf just giving a presheaf on gross objects.
Moreover, in the case of a topological space and the category of
its open sets forming a site,
the gross objects correspond to irreducible
open sets. Hence,
we have immediately the following well known fact.

\begin{corollary}\label{cor-simple}
Let $X$ be a topological space. If $X$ has a basis consisting of irreducible open sets, then
there is an equivalence between presheaves on irreducible open sets in $X$
and sheaves on $X$.
\end{corollary}

\begin{remark}\label{SetstoVect}
The definitions and the results contained in this subsection hold true if we replace
the category of sets with the one of finite dimensional $\bk$-vector spaces for a fixed field $\bk$.
These results hold more in general for categories having as objects
``esp\'{e}ces de structure alg\'{e}brique d\'{e}finie par limites projectives finie'' as in \cite{SGA4}. Instead of defining this notion, following \cite[\href{https://stacks.math.columbia.edu/tag/00YR}{Section 00YR}]{SP}, we simply list a few categories that can be used in place of the category of sets
in the definitions and statements:
abelian groups, groups, monoids, rings, modules over a ring, Lie algebras.
\end{remark}
\subsection{Locally constant sheaves}

For this subsection we consider only Grothendieck sites  $(\cC,\cT)$
where every object of $\cC$ admits a covering by gross objects.
\begin{lemma}
Let $(\cC,\cT)$ be a Grothendieck site.
Then, the equivalence $\Pre(\cC_{\mathrm{gross}})\simeq\Sh(\cC,\cT)$ associates locally constant sheaves in $\Sh(\cC,\cT)$ to presheaves in $\Pre(\cC_{\mathrm{gross}})$ mapping all arrows of $\cC_{\mathrm{gross}}$ to isomorphisms.
\end{lemma}
\begin{proof}
    Recall that the equivalence $i^\ast:\Sh(\cC,\cT)\rightarrow\Pre(\cC_{\mathrm{gross}})$ is induced by the inclusion $\cC_{\mathrm{gross}}\hookrightarrow\cC$.
    To check the statement, consider a locally constant sheaf $F\in\Sh(\cC,\cT)$. Because of the very definition of locally constant sheaf, we check that for every gross object $V$ of $\cC$, we have that $F_{|V}$ is a constant sheaf and therefore we know, by the very definition of $F_{|V}$ (see \cite[\href{https://stacks.math.columbia.edu/tag/00XZ}{Section 00XZ}]{SP}, the notation $j_V^{-1}(F)$ in place of our $F_{|V}$ is used there), that all the morphisms having target $V$ are sent by $F$ to isomorphisms. As a consequence, $i^\ast F$ has the desired form.
\end{proof}

The following theorem is the adaptation to our context of Theorems 3.6 and 3.7 of \cite{salashomotopy}.
Let  $\underline{A}$ be the sheafification of the constant presheaf on
a given Grothendieck site,
with value $A$, for a ring $A$.
\begin{theorem}
Let $(\cC,\cT)$ be a Grothendieck site
and let $\Lambda$ be a sheaf of rings on it. 
{A sheaf $M$ of $\Lambda$-modules} is quasi coherent if and only if for
every morphism between gross objects $U\rightarrow V$ we have an isomorphism
$$
M(V)\otimes_{\Lambda(V)}\Lambda(U)\rightarrow M(U)
$$
If $\Lambda=\underline{A}$
for a ring $A$, then $M$ is quasi coherent if and only if $M$ is locally constant.
\end{theorem}

\begin{proof}
For the first part, one notice that replacing in the proof of Theorem 3.6 in \cite{salashomotopy} $U_p$ and $U_q$ with two gross objects $U$ and $V$ and $q\geq p$ with an arrow between $U$ and $V$ the same proof goes through. Using the same substitution in the proof of Theorem 3.7 in \cite{salashomotopy} 
and replacing $\cO$ with $\underline{A}$ there we conclude.
\end{proof}

\begin{lemma}
Let $\mathbb{K}$ be a field and let $(\cC,\cT)$ be a Grothendieck site.
Then the category of $\underline{\mathbb{K}}$-modules
is equivalent
to the category of sheaves of $\mathbb{K}$-vector spaces on the site $(\cC,\cT)$.
\end{lemma}
\begin{proof}
    First notice that $\underline{\mathbb{K}}$ is locally constant, therefore its restriction to the category $\Pre(\cC_{\mathrm{gross}})$ is isomorphic to the constant presheaf with value $\mathbb{K}$. As a consequence, by Theorem \ref{GrossBasisThm} the category of $\underline{\mathbb{K}}$-modules on the site $(\cC,\cT)$ is equivalent to the category of presheaves of $\mathbb{K}$-vector spaces on $\cC_{\mathrm{gross}}$. As the latter is equivalent to the category of sheaves of $\bk$-vector spaces on the site $(\cC,\cT)$ we conclude.
\end{proof}

\subsection{Laplacian pairs} \label{lapl-pa-sec}
We first recall a few notions about \v{C}ech complexes and then
we define the notion of Laplacian pair.
We will need
these concepts later on, when we discuss differential operators on directed graphs.

\begin{observation}\label{Cechrmk}
Consider a Grothendieck site $(\cC, \cT)$, a finite covering
$\cU:=\lbrace U_i\rightarrow U\rbrace_{i\in I}$ and a sheaf
$F$ of abelian groups. From now on, till the end of this
section, the reader can check that we can replace ``abelian groups'' with many
of the categories listed in \ref{SetstoVect}. 
If we give a well-order to $I$, we can define two different, however \emph{homotopic}
\v{C}ech complexes. One has in degree $n$ the group:
$$
\widetilde{C}^n(\cU,F):=\oplus_{(i_0,...,i_n)\in I^{n+1}}F(U_{i_0}\times_U\cdots\times_U U_{i_n})
$$
while the other one the group
$$
C^n(\cU,F):=\oplus_{(i_0<...<i_n)}F(U_{i_0}\times_U\cdots\times_U U_{i_n}).
$$
The former is the ordinary \v{C}ech complex while the latter is called the ordered \v{C}ech complex.
Both have practical advantages and disadvantages, even though they compute the same cohomology (see \cite[\href{https://stacks.math.columbia.edu/tag/01FG}{Section 01FG}]{SP} for the case of sheaves on a topological space or \cite{Conradcech}). In this paper, ordered complexes will make the computations easier, therefore we shall always assume to use them every time we will deal with \v{C}ech complexes. We explicitly note that in this case the differentials $d_n:C^n(\cU,F)\rightarrow C^{n+1}(\cU,F)$ are given by the formula
\begin{equation}\label{eq-d}
d_n(s)_{i_0,...,i_{n+1}}=\sum_{j=0}^{n+1}(-1)^jF(\pr_{_{i_0,...,\widehat{i}_j,...,i_{n+1}}})
(s_{i_0,...,\widehat{i}_j,...,i_{n+1}})
\end{equation}
where $\pr_{_{i_0,...,\widehat{i}_j,...,i_{n+1}}}:U_{i_0}\times_U\cdots\times_U U_{i_n}\rightarrow U_{i_0}\times_U\cdots\times_U\widehat{U}_{i_j}\cdots\times_U U_{i_n}$ is the usual projection morphism and the hat denotes omission. We can give analogous definitions of \v{C}ech complexes for 
cosheaves (see \cite{BR97} or \cite{curry} for this notion) that we will call \v{C}ech cocomplexes
We will denote these complexes as $\check{C}(\cU,F)$.
\end{observation}

\begin{definition} \label{def:bisheafcplx}
Let $(\cC,\cT)$ be a Grothendieck site, $\overline{F}$, $\underline{F}$ a sheaf and a cosheaf.
We define a \textit{bisheaf complex} with respect to a finite well-ordered covering
$\cU(U)=\lbrace U_i\rightarrow U\rbrace_{i\in I}$ of an element $U$ of $\cC$, denoted as
$(\overline{F},\underline{F},\cU(U))$, to be the datum of the two maps:
$$
\begin{array}{c}
\overline{f}_{i_1<i_2<...<i_k}:\underline{F}(U_{i_1}\times U_{i_2}\times ... \times U_{i_k})
\rightarrow\overline{F}(U_{i_1}\times U_{i_2}\times ... \times U_{i_k}) \\ \\
\underline{f}_{i_1<i_2<...<i_k}:\overline{F}(U_{i_1}\times U_{i_2}\times ... \times U_{i_k})\rightarrow\underline{F}(U_{i_1}\times U_{i_2}\times ... \times U_{i_k})$, $i_1,...,i_k\in I
\end{array}
$$
\end{definition}
Notice that the introduction of a finite well-ordering on the coverings
is an immaterial choice, 
to avoid cumbersome notations, when looking at (co)homology.

Given a bisheaf complex $(\overline{F},\underline{F},\cU(U))$ we have the following maps:
$$
\begin{array}{c}
\overline{d}_n:C^n(\cU(U),\overline{F})\rightarrow C^{n+1}(\cU(U),\overline{F}) \\ \\
\underline{d}_n:C_{n+1}(\cU(U),\underline{F})\rightarrow C_n(\cU(U),\underline{F})
\end{array}
$$
where the first map is exactly the one in (\ref{eq-d}) and the second one is the analogue for the \v{C}ech cocomplex of $\underline{F}$. We also have the morphisms:

$$
\begin{array}{c}
\overline{F}_n(\cU(U)):=\oplus_{{i_1<i_2<...<i_k}}\, \overline{f}_{{i_1<i_2<...<i_k}}:C_n(\cU(U),\underline{F})\rightarrow C^n(\cU(U),\overline{F}) \\ \\
\underline{F}_n(\cU(U)):=\oplus_{i_1<i_2<...<i_k}\,
\underline{f}_{{i_1<i_2<...<i_k}}:C^n(\cU(U),\overline{F})\rightarrow C_n(\cU(U),\underline{F})
\end{array}
$$

\begin{definition}
Let $(\cC,\cT)$ and $(\overline{F},\underline{F},\cU(U))$ be a Grothendieck site and
a bisheaf complex respectively. We define the \textit{upper and lower Laplacians}
as the morphisms $\overline{\Delta}_n$, $\underline{\Delta}_n:$
$C^n(\cU(U),\overline{F})\rightarrow C^n(\cU(U),\overline{F})$:
$$
\begin{array}{c}
\overline{\Delta}_n(\overline{F},\underline{F}, \cU(U)):=
\overline{F}_n(\cU(U))\circ\underline{d}_{n}\circ \underline{F}_{n+1}(\cU(U))\circ\overline{d}_n
\\ \\ 
\underline{\Delta}_n(\overline{F},\underline{F}, \cU(U)):=\overline{d}_{n-1}\circ\overline{F}_{n-1}(\cU(U))\circ\underline{d}_{n-1}\circ \underline{F}_{n}(\cU(U))

\end{array}
$$
Finally, we define the Laplacians of the bisheaf complex to be
the morphisms
\begin{equation}\label{la-def}
\Delta_n:=\underline{\Delta}_n-\overline{\Delta}_n
\end{equation}
\end{definition}

In the context of many applications (for example when considering sheaves on Grothendieck topologies associated to digraphs) we are interested to consider a less general notion of the one of a bisheaf complex with respect to a finite well-ordered covering. Because of its usefulness in the sequel, we introduce it explicitly.

\begin{definition}\label{def:laplacianpairs}
Let $(\cC,\cT)$ be a Grothendieck site, $\overline{F}$, $\underline{F}$ a sheaf and a cosheaf as above. We define a \textit{weak Laplacian pair} with respect to a finite well-ordered covering $\cU(U)=\lbrace U_i\rightarrow U\rbrace_{i\in I}$, of an element $U$ of $\cC$, denoted as $(\overline{F},\underline{F},\cU(U))$, to be the datum of maps $f_{i}:\underline{F}(U_i)\rightarrow\overline{F}(U_i)$ and $f_{i<j}:\overline{F}(U_i\times_U U_j)\rightarrow\underline{F}(U_i\times_U U_j)$.
We call \textit{Laplacian pair} a weak Laplacian pair where the maps $f_i$ and $f_{i<j}$ are all isomorphisms.
We call \textit{strong Laplacian pair} 
a weak Laplacian pair with $\overline{F}(U_i)=\underline{F}(U_i)$,
$\overline{F}(U_i\times_U U_j)=\underline{F}(U_i\times_U U_j)$ and
the maps $f_i$ and $f_{i<j}$ are all identities.
\end{definition}

As in the case of bisheaves, given a weak Laplacian pair $(\overline{F},\underline{F},\cU(U))$
the following morphisms are well defined:
\begin{equation}\label{d-lapl-p}
\overline{d}:C^0(\cU(U),\overline{F})\rightarrow C^1(\cU(U),\overline{F}),
\qquad
\underline{d}:C_1(\cU(U),\underline{F})\rightarrow C_0(\cU(U),\underline{F})
\end{equation}
coming from the \v{C}ech (co)complexes and
\begin{equation} \label{F-lapl-p}
\begin{array}{c}
F_0(\cU(U)):=\oplus_if_{i} :C_0(\cU(U),\underline{F})\rightarrow C^0(\cU(U),\overline{F})\\ \\
F_1(\cU(U)):=\oplus_if_{i<j}:
C^1(\cU(U),\overline{F})\rightarrow C_1(\cU(U),\underline{F})
\end{array}
\end{equation}

\begin{definition}
    Let $(\cC,\cT)$ and $(\overline{F},\underline{F},\cU(U))$ be a Grothendieck site and a weak Laplacian pair respectively. We define the \textit{\v{C}ech Laplacian} as the morphism $$\Delta(\overline{F},\underline{F}, \cU(U)):=F_0(\cU(U))\circ\underline{d}\circ F_1(\cU(U))\circ\overline{d}:C^0(\cU(U),\overline{F})\rightarrow C^0(\cU(U),\overline{F})$$
\end{definition}

\section{Sheaves on graphs}\label{sec:shgraph}
We now specialize the previous treatment to the case of graphs. 

\subsection{Grothendieck sites on graphs}
We are now ready to define the Grothendieck topologies for directed graphs we are interested in.
To this end, we need to introduce some special morphisms between graphs.

\begin{definition}
Let be $f:G\rightarrow H$ a morphism between directed graphs. We say that $f$ is
\begin{itemize}
\item[1)] an \emph{embedding} if both $f_0:G_0\rightarrow H_0$ and $f_1:G_1\rightarrow H_1$ are injective ($G$ is a subgraph of $H$ via $f$).
\item[2)] a \emph{covering map} (\emph{\'{e}tale map}) if  for each $v\in G_0$, $f$ induces bijections
(injective maps) $h_G^{-1}(v)\cong h_H^{-1}(f(v))$ and $t_G^{-1}(v)\cong t_H^{-1}(f(v))$ \cite{FrTop}, \cite{stallings}.
\end{itemize}
 We say that a morphism $f$ between undirected graphs in $\Graphs$
is an \textit{embedding} (\'{e}tale map, covering map)
if the morphism $\Theta(f)$ obtained using the functor of Proposition \ref{prop-graphffemb}
is an embedding (\'{e}tale map, covering map).
\end{definition}

These morphisms, whose geometric meaning resembles the ones of the analogue morphisms defined in algebraic
topology or algebraic geometry, enjoy the following properties.

\begin{proposition}\label{prop-embetprop}
The following hold:
\begin{itemize}
\item[1)] covering maps, \'{e}tale maps and embeddings are stable under base change,
\item[2)] any \'{e}tale morphism $f:G\rightarrow H$ factors as $f=\pi\circ\iota:G\xrightarrow{\iota} M\xrightarrow{\pi} H$ where $\iota$ is an embedding and $\pi$ is a covering map.
\end{itemize}
\end{proposition}
\begin{proof}
The proof of 1) is a simple check, while the proof of 2) can be found as Lemma 2.9 in \cite{FrTop}, see also
\cite{stallings}.
\end{proof}

\begin{definition}\label{Subgraphtop}
We can define the following Grothendieck topologies and sites on graphs
and semisimplicial sets.

\noindent
1) \textit{Standard topology on digraphs.} Let $G$ be a digraph, i.e. $G\in\diGraphs$. Consider the category $\open(G)$ having as objects the subgraphs of $G$ and as arrows the inclusion morphisms. For every object $U$ of $\open(G)$ we define a collection of covers $\Cov(U)$ by declaring that a family of arrows $\mathcal{U}=\lbrace f_i:U_i\rightarrow U\rbrace_{i\in I}$ in $\open(G)$ is a covering if $\bigcup_{i\in I}f_i(U_i)=U$. As the fibre product of two subgraphs of $G$ is still a subgraph of $G$, this datum defines a Grothendieck site $(\open(G), \cT_G)$. One can check that $\open(G)_{\gross}(\cT_G)$ consists of graphs of the form:\\

\begin{center}
\begin{tikzpicture}[scale=.45]
\node (a1) at (-2.5,0) {$\bullet$};
\node at (-2.5,-0.7) {$v$};
 
\node (v) at (2.5,0) {$\bullet$};
\node (w) at (6,0) {$\bullet$};
\node at (2.5,-0.7) {$v$};
\node at (6,-0.7) {$w$};
\draw [->] (v) to (w);
\end{tikzpicture}
\end{center}
We then have isomorphisms of categories $\open(G)_{\gross}(\cT_G)\cong
\Gamma(G)_+\cong P_G$ (see Obs. \ref{digraph-poset}).
Because of Theorem \ref{GrossBasisThm} we then have that 

\begin{equation}\label{eq:stdtopgross}
\Pre(P_G)\cong\Pre(\open(G)_{\gross}(\cT_G))\simeq\Sh(\open(G),\cT_G)
\end{equation}

\noindent
2) {\textit{Standard topology on undirected graphs.}}
For an undirected graph $G$, i.e. $G\in\Graphs$ (as in \cite{godsil})
we can define a Grothendieck site $(\open(G), \cT_G)$
similarly: the objects in $\open(G)$ are subgraphs of $G$, while the arrows the inclusion morphisms.
The coverings are defined in the same way, we leave to the reader the tedious checks.
Note that in this case, the functor $\Theta: \Graphs\rightarrow\diGraphs$ of Proposition
\ref{prop-graphffemb} induces a continuous functor
$\Theta_{|\open(\Theta(G)}^\ast\open(G)\rightarrow \open(\Theta(G))$.
We have $\Pre(\cP_{G})\cong \Pre(\Gamma(G)_+)\simeq\Sh(\open(G),\cT_G)$
by Obs. \ref{graph-poset}.

\noindent
3)\textit{Standard topology on semisimplicial sets.}
Let $\open(X)$ be the category having as objects sub-semisimplicial sets of  a semisimplicial set $X$ and as arrows the inclusions. Coverings are defined similarly to the case of digraphs as jointly surjective collections of inclusion morphisms (i.e. $\cup f_i(U_i)=U$, where $f_i:U_i \lra U$ are the inclusion morphisms). Finally, one can check that the gross objects in this case are the objects of $\open(X)$ that are isomorphic to some $\Delta^n_+$.
When the simplicial set $X$ is a digraph, the resulting site is the one defined in 1).

\noindent
4) \textit{\'Etale topology on digraphs.}
  Let be $G$ a digraph, we define the \'etale Grothendieck topology
$(\Et(G), \cT_{et}(G))$ as follows.
We define the category $\Et(G)$ having as objects all the graphs $H$, such that there exists an \'etale
morphism $H\rightarrow G$, and as morphisms \'etale morphisms. As before, for every object $U$ of $\Et(G)$ we define a collection of covers $\Cov(U)$ by declaring that a family of arrows $\mathcal{U}=\lbrace f_i:U_i\rightarrow U\rbrace_{i\in I}$ is a covering if $\bigcup_{i\in I}f_i(U_i)=U$.
It follows from Proposition \ref{prop-embetprop} that this assignment defines a Grothendieck topology on $\Et(G)$ that we denote as $\cT_{et}(G)$. The inclusion functor $\open(G)\subseteq\Et(G)$ induces a continuous functor $(\open(G),\cT(G))\rightarrow (\Et(G),\cT_{et}(G))$. If $G$ is a standard undirected graph,
the functor $\Theta$ of Proposition \ref{prop-graphffemb}
induces continuous functors $(\open(G),\cT(G))\rightarrow (\open(\Theta(G)),\cT(\Theta(G)))$ and $(\Et(G),\cT_{et}(G))\rightarrow (\Et(\Theta(G)),\cT_{et}(\Theta(G)))$.
\end{definition}

Notice that the definitions above can be generalized to define
``big'' sites as in the case of algebraic geometry (see \cite{Vis}).

\begin{figure}[b]
\begin{center}
\begin{tikzpicture}[scale=.45]
\node (G) at (-12.5,-2) {$G$};
\node (a) at (-9.5,-2) {$\bullet$};
\node (b) at (-6,-2) {$\bullet$};
\node (c) at (-2.5,-2) {$\bullet$};
\node at (-9.5,-2.7) {$u$};
\node at (-6,-2.7) {$v$};
\node at (-2.5,-2.7) {$w$};
\draw [->] (a) to [out=30,in=150] (b);
\draw [->] (b) to [out=-150,in=-30] (a);
\draw [->] (a) to (b);
\draw [->] (b) to (c);

\node (H) at (-12.5,2) {$H$};
\node (a1) at (-9.5,2) {$\bullet$};
\node (b1) at (-6,2) {$\bullet$};
\node (c1) at (-2.5,2) {$\bullet$};
\node at (-9.5,1.3) {$u'$};
\node at (-6,1.3) {$v'$};
\node at (-2.5,1.3) {$w'$};
\draw [->] (a1) to (b1);
\draw [->] (b1) to (c1);
\draw [->] (a1) to [out=30,in=150] (b1);
 
\node (a2) at (2.0,2) {$\bullet$};
\node (b2) at (5.5,2) {$\bullet$};
\node at (2.0,1.3) {$u''$};
\node at (5.5,1.3) {$v''$};
\draw [->] (b2) to (a2);
\draw [->] (-12.5,1) to (-12.5,-1);
\end{tikzpicture}
\end{center}
\caption{
\'Etale non standard cover of $G$ by the disconnected digraph $H$.}
\label{fig:etale}
\end{figure}
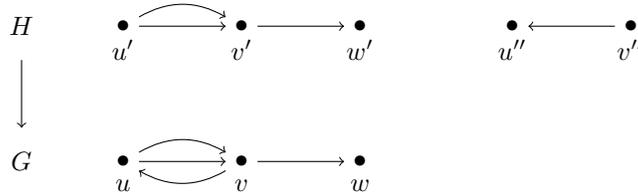

\begin{observation} Let $G \in \diGraphs$.
Every standard cover is also \'etale, however the viceversa is not true
(see Fig. \ref{fig:etale} displaying an \'etale cover which is not a standard one).
Also, any $G \in \diGraphs$ admits an \'etale cover (which may not be standard), consisting
of just one map $\{f:H\lra G\}$, where $H$ can be chosen in $\diGraphss$, as we see, again, in Fig. \ref{fig:etale}.
where $f_0:H_0 \lra G_0$, $f_0(u')=f_0(u'')=u$, $f_0(v')=f_0(v'')=v$, $f_0(w')=w$ and similarly $f_1$.
\end{observation}

\subsection{Sheaves on digraphs}\label{Sect:sheavesdigraphs}
We are now ready to explore the notion of sheaf arising in the standard topology on digraphs, as defined in item 1 of Definition \ref{Subgraphtop}. From now on, until the end of this subsection, we will assume that all the presheaves and sheaves we consider are valued in the category $\Vect_\mathbb{K}$ of finitely generated vector spaces on a field of characteristic different from 2 unless we say otherwise.
Many of our considerations and examples will hold also in the context of (pre)sheaves having values in the categories listed in \ref{SetstoVect}.

\begin{observation}\label{sheaves:digraph}
Recall that, because of Theorem \ref{GrossBasisThm},
a sheaf on a digraph $G=(E_G,V_G,h_G,t_G)$ for the standard topology is equivalent to a 
presheaf $F\in\Pre(\open(G)_{\gross}(\cT_G))$. This is the datum of:
\begin{itemize}
\item a vector space $F(v)$ for each
vertex $v\in V_G$,
\item a vector space $F(e)$ for each edge (with its endpoints) $e\in E_G$,
\item linear maps (restriction maps)
$F_{h_G(e)\leq e}: F(e)\rightarrow F(h_G(e))$,
  $F_{t_G(e)\leq e}:F(e)\rightarrow F(t_G(e))$ for each edge $e\in E_G$,
where, we write $v\leq e$ to mean
that $v$ is a vertex of the edge $e$ (note that we only have one map $F(e)\to F(v)$ if $e$ is a self loop based at $v$).
\end{itemize}

Indeed, the gross objects for $\open(G)$ are precisely
the vertices and edges of $G$. This is exactly the definition given in \cite{FrTop}.
A similar observation holds for a sheaf on a graph $G\in\Graphs$, i.e. in the topology
as in Def. \ref{Subgraphtop}; the gross objects of such topology are the vertices and the edges
(with their endpoints) of $G$. Hence, given a presheaf $F \in  \Pre(\open(G)_{\gross}(\cT_G))$
we denote the restriction maps as $F_{v\leq e}: F(e)\rightarrow F(v)$.
\end{observation}

Finally, recall that in the previous section, we defined the functor $\Theta:\stGraphs\rightarrow\diGraphs$ that gives us a way to see a simple undirected graph as a directed graph via its image under $\Theta$.
In the next observation we explore the relation between the sheaves
on the site $(\open(G),\cT_{G})$ and those on $(\open(\Theta(G)),\cT_{\Theta(G))})$.
\begin{observation}\label{obs:moveshvs}
Let $G$ be a graph in $\stGraphs$.
\noindent
1) The functor $\Theta$ defined in Section \ref{sec:undgr} is continuous,
  therefore, given any sheaf $F$ on $(\open(\Theta(G)),\cT_{\Theta(G))})$, we
can pull it back, functorially, to a sheaf $\Theta^p(F)$ on $(\open(G),\cT_{G})$.
More precisely, for any subgraph $H$ in $G$, $\Theta^p(F)(H):= F(\Theta(H))$.
\noindent
2) The functor $\Theta$ is fully faithful (cfr. Theorem \ref{prop-graphffemb}), however, it is not true in general that for a given $G\in\stGraphs$, $\Sh(\open(G),\cT_G)$ is equivalent to the category $\Sh(\open(\Theta(G)),\cT_{\Theta(G))})$ via $\Theta_{|\open(G)}^p$ as the next example shows.
The sheaves on the undirected graph $G\in \stGraphs$ in Fig. \ref{fig-g1} are simpler than sheaves on the digraph $H$:
indeed
$\Sh(\open(G),\cT_G)$ is not equivalent to $\Sh(\open(\Theta(G)),\cT_{\Theta(G))})$. For example, note that,
while $H=\Theta(G)$ is not gross in $(\open(\Theta(G)),\cT_{\Theta(G}))$,
    $G$ is indeed gross in $(\open(G),\cT_G)$.
    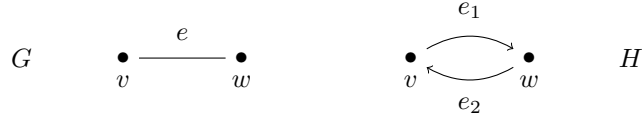
\begin{figure}[h]
    \begin{center}
    \begin{tikzpicture}[scale=.45]
    \node (G) at (-9,0) {$G$};
    \node (v) at (-6,0) {$\bullet$};
    \node (w) at (-2.5,0) {$\bullet$};
    \node at (-6,-0.7) {$v$};
    \node at (-2.5,-0.7) {$w$};
    \node at (-4.25,0.7) {$e$};
    \draw [-] (v) to (w);

    \node (H) at (9,0) {$H$};
    \node (v) at (2.5,0) {$\bullet$};
    \node (w) at (6,0) {$\bullet$};
    \node at (2.5,-0.7) {$v$};
    \node at (6,-0.7) {$w$};
    \node at (4.25,1.4) {$e_1$};
    \node at (4.25,-1.4) {$e_2$};
    \draw [->] (v) to [out=30,in=150] (w);
    \draw [->] (w) to [out=-150,in=-30] (v);
    \end{tikzpicture}
    \end{center}
    \caption{The undirected graph $G$ and its image via $\Theta$}\label{fig-g1}
\end{figure}
\end{observation}

Now, let us assume we have a sheaf $F$ on $G\in\stGraphs$, with respect to the standard topology.
It is natural to ask whether we can define a sheaf $F'$ on $\Theta(G)$ so that $\Theta^p(F')$
is isomorphic to $F$, thus showing that $\Theta^p$ is essentially surjective
Moreover,
we could also ask if $\Theta^p$ is full with a section i.e. we have a
faithful functor $\hat\Theta$ such that $\Theta^p\circ\hat\Theta$ is
isomorphic to the identity functor
$\mathrm{Id}: \Sh(\open(G), \cT_G)\rightarrow \Sh(\open(G), \cT_G)$.
Furthermore, we also ask
if there is a faithful embedding
$\Sh(\open(G), \cT_G) \rightarrow \Sh((\open(\Theta(G)), \cT_{\Theta(G)})$.
This last property would allow us to study sheaves on
$G$ by identifying them with certain sheaves on $\Theta(G)$.
As we shall
see in our next observation, the natural functor (adjoint of $\Theta^p$)
does not have the required properties, however, with a small modification,
we can define a functor $\hat \Theta$ achieving what is stated above.

\begin{observation}\label{obs:moveshvs2}
1) The functor $\Theta^p$ has a left adjoint $\Theta_p$ (see \cite[\href{https://stacks.math.columbia.edu/tag/00VE}{Lemma 00VE}]{SP}), thus one could use the latter to define a sheaf $\Theta_p(F)$ on $\Theta(G)$, however the unit $\eta: 1_{\Sh(\open(G),\cT_G)} \rightarrow \Theta^p\Theta_p$ is not an isomorphism, in general.
Indeed, consider again the graph $\alertb{G}$ depicted in Fig. \ref{fig-g1}. As we already know, to give a sheaf on $G$ it suffices to give a presheaf on its gross objects. Thus we have three vector spaces $F(v), F(w), F(e)$ and two maps $F_{v\leq e}, F_{w\leq e}$. The sheaf $\Theta_p(F)$ is the datum of the vector spaces $\Theta_p(F)(e_1):=F(e)$, $\Theta_p(F)(e_2):=F(e)$, $\Theta_p(F)(v):=F(v)$, $\Theta_p(F)(w):=F(w)$ and maps $\Theta_p(F)_{v\leq e_1}=\Theta_p(F)_{v\leq e_2}:=F_{v\leq e}$, $\Theta_p(F)_{w\leq e_1}=\Theta_p(F)_{w\leq e_2}:=F_{w\leq e}$. We leave as a check to the reader that in fact the vector space of global sections $\Theta_p(F)(\Theta(G))$, is equal to $F(e)$ if only if $F_{v\leq e}$ and $F_{w\leq e}$ are injective.
Therefore, unless we are in this last hypothesis, we do not get
that $\Theta^p\Theta_p(F)\cong F$.
This statement is actually more general.
Given $G\in\stGraphs$, we take the covering $G=\bigcup G_i$ by its
subgraphs formed by a single edge and its endpoints, which are
gross objects for the standard topology. Then, given a sheaf $F$ on $G$,
$\Theta^p\Theta_p(F)\cong F$ if and only if 
the maps $F_{v\leq e}, F_{w\leq e}$ are injective for each $G_i$,
vertices $v$ and $w$ and edges $e$.
\noindent
2)
As we are considering sheaves valued in $\Vectk$, we want to define a functor $\hat\Theta$ from $\Sh(\open(G),\cT_{G})$ to $\Sh(\open(\Theta(G)),\cT_{\Theta(G))})$ such that, for a given sheaf $F$ on $G$, $\Theta^p\hat\Theta(F)\cong F$. We show how to define it on the graph in Fig. \ref{fig-g1}. For a general $G\in\stGraphs$, general construction follows using the same reasoning we used at the end of item 1 in this observation, {provided we make a choice of a distinguished edge for every pair of edges in $\Theta(G)$ connecting two distinct vertices}. Let $\gamma: F(e)\rightarrow F(v)\oplus F(w)$ be the morphism defined by $x \mapsto (F_{v\leq e}(x), F_{w\leq e}(x))$ and set $\hat\Theta(F)(e_1):=F(e)$, $\hat\Theta(F)(e_2)=F(e)/\ker \gamma$, $\hat\Theta(F)(v):=F(v)$, $\hat\Theta(F)(w)=F(w)$. As for the restriction maps, notice that both $F_{v\leq e}$ and $F_{w\leq e}$ are defined on the quotient $F(e)/\ker \gamma$ by construction. Thus we set $\hat\Theta(F)_{v\leq e_1}=\hat\Theta(F)_{v\leq e_2}:=F_{v\leq e}$, $\hat\Theta(F)_{w\leq e_1}=\hat\Theta(F)_{w\leq e_2}:=F_{w\leq e}$ {(the definition for self loops is clear)}.
With this definition, we have $\hat\Theta(F)(\Theta(G))\cong F(e)$, by the very definition
  of sheaf, and therefore $\Theta^p\hat\Theta(F)\cong F$. Moreover the functor $\hat\Theta$
  is faithful and satisfies $\Theta^p\circ\hat\Theta\cong \mathrm{Id}$, where
  $\mathrm{Id}: \Sh(\open(G), \cT_G)\lra \Sh(\open(G), \cT_G)$.
\end{observation}

Finally, note that given an oriented undirected graph
we have an equivalence
of categories between sheaves on it with respect to the standard topology and the category of
sheaves on it with respect to the standard topology when considered as a directed graph by
using its orientation.
In the next sections we will establish a link between our treatment
and the sheaves on graphs as in \cite{FrTop, HG19}
and in the context of graph neural networks in \cite{bodnar}.

\subsection{Graphs, finite topological spaces and preordered sets}
\label{sec-preorder}
A topological space is called \textit{finite} if it consists of a finite number of points.
For a given finite space $X$ and a point $p\in X$, we define
$$
U_p:=\textnormal{ smallest open subset of }X\textnormal{ containing }p
$$
$$
C_p:=\overline{p}=\textnormal{ smallest closed subset of }X\textnormal{ containing }p
$$
Now, given a finite topological space $X$, we can define the structure of a (finite)
preorder on $X$ by setting $p\leq q$ if and only if
$U_p\supseteq U_q$ (equivalently, $p\in \overline{q}$).
Conversely, given a {(not necessarily finite)} preorder $\leq$ on a set $P$, 
we can see $P$ as a topological space,
whose topology is generated by the basis consisting of the following open sets:
$$
U_p:=\lbrace q\in P\: |\: q\geq p\rbrace\qquad p\in P
$$
This topology is called the \textit{Alexandrov topology} \cite{alex}.

\begin{notation}\label{notation:alexandrovtop}
Whenever necessary, to avoid confusion,
we shall use a different notation for a \textit{preorder} $P$ (i.e. a preordered set)
and the topological space $A(P)$ associated with it, as above.
Similarly, for a finite topogical space $X$ we will write $P(X)$ for
the preorder associated with it as above.
We note explicitly that if $X$ is a finite topological space and
$P(X)$ is the associated preoder, then $X=A(P(X))$. We denote by
$\mathrm{Open}(A(P))$ the category of the
open sets of a preorder $P$ viewed as topological space.
$(\mathrm{Open}(A(P)), \cT(A(P)))$ (coverings with arrows given by inclusions) is a site:
sheaves for this site are the usual sheaves on the topological space $A(P)$.
\end{notation}

We have the following result, see \cite{salas} for a clear review of these statements.

\begin{theorem}\label{Thm:alextoppos}
The following statements hold:
\begin{itemize}
\item[1)] The above construction defines an equivalence of categories between finite
topological spaces $\mathrm{FTop}$ and finite preordered sets $\mathrm{PreSets}$
given by the functors:
$$
P: \mathrm{FTop}\lra \mathrm{PreSets}, \qquad A:\mathrm{PreSets}\lra \mathrm{FTop}
$$

\item[2)] A finite topological space $X$ is $T_0$ (i.e. different points have different closures)
if and only if the preorder relation $\leq$ induced by the topology is antisymmetric
i.e. $X$ is a poset.
\end{itemize}
\end{theorem}

\begin{remark}\label{rmk:irred}
Notice that, for a given finite topological space $X$, the irreducible open and
closed sets are $U_p$ and $C_p$, $p\in X$.
The same holds true if $X$ is a topological space of the form $A(P)$
for some (not necessarily finite) preorder $P$.
\end{remark}

\begin{observation}
If $X$ is a semisimplicial set, we have $P_X$ as in Obs. \ref{digraph-poset},
hence 
$A(P_X)$ is defined as above. We are interested
in the case $A(P_G)$,
$G\in\diGraphs$.
\end{observation}

\begin{example}\label{graph-alex}
Let $G=(V,E)$ be in $\Graphs$ (note that in general we do not
have $E\not\subset V \times V$). In Observation \ref{graph-poset} we have defined a poset $\mathcal{P}_G$ associated to $G$. This is the poset having as elements the vertices and the edges of $G$  and such that $x\leq y$ if and only if $x=y$ or $x$ is a vertex of the edge $y$.
The topology of the space $A(\cP_G)$ is then generated by the following basis of open sets
(see Figure \ref{fig:openstar}):
\begin{itemize}
    \item $U_v=\lbrace e\in E\: |\: v\leq e\rbrace\cup\{v\}$, that is the open star of $v$, for each vertex $v\in V$,
\item $U_e=\lbrace e\rbrace$, i.e. the edge $e$, without its vertices, for each $e\in E$.
\end{itemize}

For $G\in\stGraphs$, we can either consider the topological space
$A(\mathcal{P}_G)$ or the one associated to its image through $\Theta$, $A(P_{\Theta(G)})$,
which are in general different.
There is a continuous open map $\phi: A(P_{\Theta(G)})\rightarrow A(\mathcal{P}_G)$
naturally induced by $\Theta$,
defined as follows.
For any vertex $v$ of $\Theta(G)$, $\phi(v)=v$. 
For every edge $v\rightarrow w$ in $\Theta(G)$, by construction,
there is a single edge $e$
in $G$: set $\phi(v\rightarrow w)=\phi(w\rightarrow v)= e$.

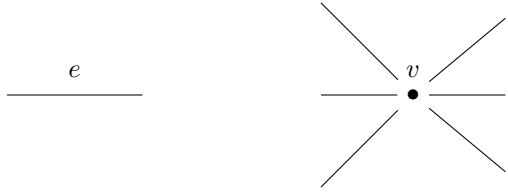
\begin{figure}[h]
\begin{center}
\begin{tikzpicture}[scale=.45]

\draw [-] (-6,0) -- (-2,0);
\node (edge) at (-4,0.7) {$e$};

\node (v) at (6,0) {$\bullet$};
\node (vertex) at (6,0.7) {$v$};
\node (a) at (3,-3) {};
\node (b) at (9,-2.5) {};
\node (c) at (3,3) {};
\node (d) at (9,2.5) {};
\node (e) at (9,0) {};
\node (f) at (3,0) {};
\draw [-] (a) -- (v);
\draw [-] (b) -- (v);
\draw [-] (c) -- (v);
\draw [-] (d) -- (v);
\draw [-] (e) -- (v);
\draw [-] (f) -- (v);
\end{tikzpicture}
\end{center}
\caption{Irreducible open sets for the topology in Ex. \ref{graph-alex}.}\label{fig:openstar}
\end{figure}
\end{example}

\begin{definition}
Given a finite topological space $X$, the \textit{dual space} $\hat{X}$
is the finite topological space
with the same underlying set as $X$, but with the opposite topology
(the open subsets are the closed subsets of $X$).
\end{definition}

\begin{example}
Let $G\in\diGraphs$ and consider the topological space $A(P_G)$.
Any vertex $v$ is a closed set, and 
edges with end vertices are also closed sets.
Moreover, if $A(P_G)$ is finite, i.e. $V_G$, $E_G$ finite,
its dual topology gives rise to a finite topological space whose associated Grothendieck site is
isomorphic to the one in Def. \ref{Subgraphtop}.
\end{example}

\begin{definition}
The \textit{dimension} of a finite topological space $X$ is defined as as the maximum of the lengths of the irreducible closed subsets of $X$, that is, the maximum of the length of the chains of points $p_0<p_1\dots<p_n$ where $<$ is the partial
order relation on $X$ induced by its topology ($a<b$ if $a\leq b$ and $a\neq b$).
\end{definition}

Finally, recall that every preordered set $(P,\leq)$ can be considered as a category having as objects the elements of $P$ and as arrows $x\rightarrow y$ if $x\leq y$.

\subsection{Sheaves on graphs as preordered sets}\label{section:sheavesposet}
In the previous section, given a graph $G\in\diGraphs$ we
associate to it a topological space via the Alexandrov topology, i.e.
via a preorder. This allows us to define sheaves on directed graphs and, more generally,
on any preorder $P$. We shall now describe the links and differences between such sheaves and those
relative to the sites defined in Sec. \ref{Sect:sheavesdigraphs}. 

\begin{observation}\label{Alextop}
For a given preorder $P$, we can consider the site $(\open(A(P)),$ $\cT_{A(P)})$,
where the gross 
objects are the open sets are $U_p$, $p\in P$.  These are the irreducible sets for
the Alexandrov topology, as we remarked above. In addition,
we notice that $p\leq q$ implies that $U_q\subseteq U_p$. As the sets $U_p$ are a base for
the Alexandrov topology, we then obtain using
Theorem \ref{GrossBasisThm} or Corollary \ref{cor-simple}:
\begin{equation}\label{eq:posetsheavesopen}
\Sh(\open(A(P)),\cT_{A(P)})\simeq\Pre(P^{op})=\Fun(P,\Sets)
\end{equation}
that is, to assign a sheaf on a preorder, it is enough to specify a presheaf on the open $U_p$'s.
This result was also obtained in Sec. 4 in \cite{curry} using different techniques and
holds also by replacing the category $\Sets$ with any category {listed in \ref{SetstoVect}}. 

As a special case, if $G$ is in $\diGraphs$, we can build the Grothendieck site $(\open(A(P_G)),\cT_{A(P_{G}))})$. We recall that the gross objects of $(\open(A(P_G)),\cT_{A(P_{G}))})$ are the edges $\{e\}=U_e$ and the \textit{open stars} $U_v$ associated with vertices $p\in V_G$ that is the union of the vertex $p$ and all the edges connected to it (see Fig. \ref{fig:openstar}).
\end{observation}

\begin{observation}\label{bod-sh}
Let $G\in \diGraphs$. By Remark \ref{Alextop}, we see that giving a sheaf on $A(P_G)$ 
is equivalent to give a presheaf $F\in\Pre(P_G^{op})$. This consists of:
\begin{itemize}
    \item an assignment of a vector space $F(U_v)$ for each vertex $v\in V_G$.
    \item  an assignment of a vector space $F(U_e)$ for each edge $e\in E_G$.
\item  an assignment of linear maps
$F_{v\leq e}:F(U_v)\rightarrow F(U_e)$ for each $v\leq e$.
\end{itemize}
An analogous observation holds true if $G\in\Graphs$. 
More is true:
the poset associated to a digraph $G\in\diGraphs$ is isomorphic to the poset associated to the undirected graph obtained from $G$ by forgetting the directions of the arrows. As a consequence, we get that the category of sheaves on a given digraph $G\in\diGraphs$ with respect to the topology given by $P_G$ is equivalent to the category of sheaves on the undirected graph obtained from $G$ by forgetting the directions of the arrows with respect to the topology induced by the poset of this undirected graph. For an undirected graph $G$ this notion of sheaves on graphs is the same as the one introduced by Bodnar et al. in \cite{bodnar, curry, HG19}, employed also in the context of machine learning. To be precise, they build their theory in the context of {\sl cellular sheaves}
(see Section \ref{sec:cellsh}), that are sheaves on posets associated with regular cell complexes.
Indeed,
we can view $G\in \Graphs$ as a cellular cell complex by taking the poset associated
to it. We then obtain the poset $P_G$ as defined in Example \ref{graph-alex}
and therefore we reach the same notion of sheaf as we defined above.
\end{observation}

We now wish to address the questions as in Obs. \ref{obs:moveshvs},
\ref{obs:moveshvs2} in the present setting.

\begin{observation}\label{obs:moveshvalex}
Let be $G\in\stGraphs$.
\noindent 1)
Analogously to what we have seen for the standard topology in Observation \ref{obs:moveshvs}, since the functor $\Theta$ defined in Section \ref{sec:undgr} is continuous, given any sheaf $F$ on $(\open(A(P_{\Theta(G)})),\cT_{A(P_{\Theta(G)})})$, one can get a sheaf $\Theta^p(F)$ on $(\open(A(\mathcal{P}_G)),\cT_{A(\mathcal{P}_{G}))})$. Consider the map $\phi$ we introduced in Example \ref{Alextop}. One could consider the usual constructions of inverse image $\phi^*$ and direct image $\phi_*$ of sheaves via $\phi$. It is easy to verify that, given a sheaf $F$ on $G$, $\phi^*(F)\cong \Theta_p(F)$ and on the other hand, given a sheaf $F'$ on $\Theta(G)$, $\phi_*(F')\cong\Theta^p(F')$.

\noindent 2)
In general, given $G\in\stGraphs$, the
category $\Sh(\open(A(\mathcal{P}_G)),\cT_{A(\mathcal{P}_{G}))})$ is not
equivalent to the category $\Sh(\open(A(P_{\Theta(G)})),\cT_{A(P_{\Theta(G)})})$
via $\Theta_{|\open(A(\mathcal{P}_G))}^p$.

    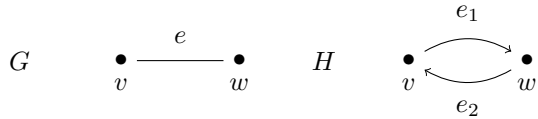
\begin{figure}[h]
    \begin{center}
    \begin{tikzpicture}[scale=.45]
    \node (G) at (-9,0) {$G$};
    \node (v) at (-6,0) {$\bullet$};
    \node (w) at (-2.5,0) {$\bullet$};
    \node at (-6,-0.7) {$v$};
    \node at (-2.5,-0.7) {$w$};
    \node at (-4.25,0.7) {$e$};
    \draw [-] (v) to (w);

    \node (H at (9,0) {$H$};
    \node (v) at (2.5,0) {$\bullet$};
    \node (w) at (6,0) {$\bullet$};
    \node at (2.5,-0.7) {$v$};
    \node at (6,-0.7) {$w$};
    \node at (4.25,1.4) {$e_1$};
    \node at (4.25,-1.4) {$e_2$};
    \draw [->] (v) to [out=30,in=150] (w);
    \draw [->] (w) to [out=-150,in=-30] (v);
    \end{tikzpicture}
    \end{center}
    \caption{The undirected graph $G$ and its image $H=\Theta(G)$ via $\Theta$}\label{fig-g2}
\end{figure}
\end{observation}

\begin{observation}\label{obs:moveshvalex2}
Now, let us assume we have a sheaf $F$ on $G \in  \Graphs$.
As we have already mentioned, the functor $\Theta^p$ has a left adjoint $\Theta_p$. In this case, in contrast with what happens for the standard topology, the unit 
$$\eta: 1_{\Sh(\open(A(\mathcal{P}_G)),\cT_{A(\mathcal{P}_{G}))})} \rightarrow \Theta^p\Theta_p$$ is indeed an isomorphism. It follows that we can use $\Theta_p$ to define a sheaf $\Theta_p(F)$ on $\Theta(G)$ such that $\Theta^p\Theta_p(F)\cong F$. Here we just show the latter condition.
Indeed, consider again the graph depicted in Fig. \ref{fig-g2}. As we already know, to give a sheaf on $G$ it suffices to give a presheaf on its gross objects. Thus we have three vector spaces $F(U_v), F(U_w), F(U_e)$ and two maps $F_{v\leq e}, F_{w\leq e}$. The sheaf $\Theta_p(F)$ is the datum of the vector spaces $\Theta_p(F)(U_{e_1}):=F(U_e)$, $\Theta_p(F)(U_{e_2}):=F(U_e)$, $\Theta_p(F)(U_v):=F(U_v)$, $\Theta_p(F)(U_w):=F(U_w)$ and maps $\Theta_p(F)_{v\leq e_1}=\Theta_p(F)_{v\leq e_2}:=F_{v\leq e}$, $\Theta_p(F)_{w\leq e_1}=\Theta_p(F)_{w\leq e_2}:=F_{w\leq e}$. We leave as a check to the reader that for this topology, the vector space of global sections $\Theta_p(F)(\Theta(G))$, is always equal to $F(G)$. Therefore we do get that $\Theta^p\Theta_p(F)\cong F$.
Following the same reasoning of Observation \ref{obs:moveshvs2}, the result above can be extended to any graph $G$, if one considers a covering of $G$ by its open stars. This tells us that there is a faithful embedding $\Sh(\open(A(\mathcal{P}_G)),\cT_{A(\mathcal{P}_{G}))})\rightarrow \Sh(\open(A(P_{\Theta(G)})),\cT_{A(P_{\Theta(G)})})$. {As a consequence}, here we are in the best scenario and we can study sheaves on $G$ by identifying them with certain shaves on $\Theta(G)$
(this holds in general for $G\in\stGraphs$).
\end{observation}

Now we can establish a relation between sheaves on graphs with the standard Grothendieck topology and those with the Alexandrov topology. {From now until the end of the section we shall assume that all the digraphs we consider are finite and that all the vector spaces we consider are finitely generated.}

Let $G\in\diGraphs$. Consider a sheaf $F\in\Sh(\open(G),\cT_G)$. Let $P_G$ be the preorder on $G$ as in Observation \ref{digraph-poset} and call $F_p\in\Pre(\open(G)_{gross})$ its restriction to the category of gross objects. 

\begin{theorem}\label{thm:dualityvect}
The duality functor $\vee:\Vectk\rightarrow \Vectk,\quad V\mapsto V^\vee=\Hom(V,\bk)$. 
    $\vee$ induces a duality $$(-)^\vee:\Pre(\open(G)_{gross})\simeq\Pre(P_G^{op})$$As a consequence, we have a duality (i.e. an anti-equivalence) $$\Sh(\open(G),\cT_G)\simeq \Sh(\open(A(P_G)),\cT_{A(P_G)})$$
\end{theorem}
\begin{proof}
The proof follows by putting together Definition \ref{Subgraphtop}, Observation \ref{sheaves:digraph}, and Observation \ref{Alextop}.
    
\end{proof}

\begin{remark}
The previous theorem clarifies the link between sheaves on graphs as defined by Curry \cite{curry},
Hansen-Ghirst \cite{HG19}, Bodnar et al. \cite{bodnar}, etc. on one side, and
Friedman \cite{FrTop} on the other. Notice that the first group of authors works with undirected graphs, while the theorem above is stated in terms of directed graphs. However, on the one hand we observe that the same exact statement holds true for the category
$\Graphs$ (i.e. undirected graphs). On the other hand, recall that we have the functor $\Theta^p$ that, given a $G\in\Graphs$ and a sheaf $F$ on $\Theta(G)$, allows us to construct a sheaf on $G$ (cfr. Observations \ref{obs:moveshvs} and \ref{obs:moveshvalex}). This gives the following diagram that is commutative up to isomorphism.
    $$\xymatrix{
    \Sh(\open(\Theta(G)),\cT_{\Theta(G))}) \ar[r]^-\simeq\ar[d]^-{\Theta^p} &  \Sh(\open(A(P_{\Theta(G))})),\cT_{A(P_{\Theta(G))})}) \ar[d]^-{\Theta^p} \\
    \Sh(\open(G),\cT_G) \ar[r]^-\simeq &  \Sh(\open(A(\mathcal{P}_G)),\cT_{A(\mathcal{P}_G)}) \\
    }$$
    The commutativity follows from the fact the the direct sum commutes with $\Hom(-,\K)$ and that $F(e)=F(U_e)$ by construction.
\end{remark}
\subsection{Homology and cohomology on directed graphs}\label{sec:hocohlapl}
We
study the notions of homology and cohomology for sheaves on a digraph with
respect to the standard topology and to the one defined by the poset of the digraph (Alexandrov topology).

We assume that all the (co)sheaves we consider
{have values in the category of finitely generated $\bk$-vector spaces} $\Vectk$,
although we could consider some of the categories listed in Obs. \ref{SetstoVect} in place of $\Vectk$.
We also assume $G\in\diGraphs$ to be a finite digraph.

\begin{definition}
For every presheaf of vector spaces $F$ on $\open(G)$ we define
$$
C_1(G,F):=\oplus_{e\in E_G}F(e), \qquad C_0(G,F):=\oplus_{v\in V_G}F(v)
$$
Paralleling \cite{FrTop}, Sec. 1,
the restriction maps can be bundled together to maps 
$$
\begin{array}{ccc}
\widetilde{d_h}:=\oplus_{e\in E_G}F_{h_G(e)\leq e}:C_1(G,F)&\rightarrow& C_0(G,F) \\ \\
\cdots\oplus x\oplus\cdots &\mapsto & \cdots\oplus F_{h_G(e)\leq e}(x)\oplus\cdots \\ \\

\widetilde{d_t}:=\oplus_{e\in E_G}F_{t_G(e)\leq e}:C_1(G,F)&\rightarrow &C_0(G,F) \\ \\
\cdots\oplus x\oplus\cdots&\mapsto& \cdots\oplus F_{t_G(e)\leq e}(x)\oplus\cdots
\end{array}
$$
We obtain the chain complex $C_\ast(G,F)$
$$
\widetilde{d}:=\widetilde{d_h}-\widetilde{d_t}:C_1(G,F)\rightarrow C_0(G,F)
$$
and we denote the homology groups of this complex as $H_1(G,F):=\mathrm{ker}(\widetilde{d})$ and
$H_0(G,F):=\mathrm{coker}(\widetilde{d})$ respectively.
\end{definition}

\begin{observation}\label{RMK:sheafcosheaf}
1) The complex $C_\ast(G,F)$,
once we choose an ordering on the vertices of $G$ and view it as a cochain complex,
is the \v{C}ech complex of $F$ with respect to the covering $\cU$ given by the edges. In other words, if we see the chain complex $C_\ast(G,F)$ as a cochain complex, we get exactly the complex $\check{C}(\cU',F)$ where $\cU'$ is the covering comprised of the closed edges of $G$.
Accordingly, we have in this case that $H_1(G,F)\cong F(G)$, since
$H_1(G,F)\cong H^0(G,F)$.

\medskip\noindent
2) Similarly, if we consider a sheaf $\widehat{F}\in\Sh(\open(A(P_G)),\cT_{A(P_G)})$,
we can set $C^1(G,\widehat{F}):=\oplus_{e\in E_G}\widehat{F}(U_e)$ and $C^0(G, \widehat{F})
:=\oplus_{v\in V_G}\widehat{F}(U_v)$ and get a cochain complex
$C^0(G,\widehat{F})\rightarrow C^1(G,\widehat{F})$, denoted $C_\ast(G,\widehat{F})$,
whose cohomology groups we denote as $H^0(G,\widehat{F})$ and $H^1(G,\widehat{F})$.
Also in this case, this complex can be seen as the \v{C}ech complex relative to the
covering of $G$ by the open sets $U_v$, $v\in V_G$ and we have
$H^0(G,\widehat{F})\cong \widehat{F}(G)$.
\end{observation}

Recall that a sheaf $\widetilde{F}\in\Sh(\open(G),\cT_G)$ on a 
$G\in\diGraphs$ is determined by a presheaf $F\in\Pre(P_G)$.
Here we shall deviate from the notation used so far and we introduce different notations
for the sheaf $\widetilde{F}\in\Sh(\open(G),\cT_G)$ and the presheaf $F$ we get on the gross objects
to avoid confusion. By Theorem \ref{thm:dualityvect}, by applying the functor
$\Hom_{\Vectk}(-,\mathbb{K})$ to $F$ we get a functor $F^\vee\in\Pre(P_G^{op})$
determining a sheaf $\widetilde{F^\vee}\in\Sh(\open(A(P_G)),\cT_{A(P_G)})$.
One
can check that $\Hom_{\Vectk}(C_\ast(G,\widetilde{F}),\mathbb{K})\cong \check{C}(\cU, \widetilde{F^\vee})$ where $\cU$ is the covering $\{U_v\}_{v\in V_G}$ (this uses the fact that the vector spaces we consider are finite dimensional
and that we consider only finite digraphs). Notice that $\widetilde{F^\vee}$ is a priori different from the presheaf $\Hom_{\Vectk}(\widetilde{F}(-),\mathbb{K})$. However, using the fact that $\Hom_{\Vectk}(-,\mathbb{K})$ is exact, we have the following result, that we leave to the
reader as a check.

\begin{proposition}
Let the notation be as above. If  $F'\in\Sh(\open(G),\cT_G)$, then  $F'^\vee:=\Hom_{\Vectk}(F,\mathbb{K})$
is a cosheaf with respect to the site $(\open(A(P_G)),\cT_{A(P_G)})$. Moreover
$$
\begin{array}{c}
\Sh(\open(G),\cT_G)\cong\coSh(\open(A(P_G)),\cT_{A(P_G)}) \\ \\
\coSh(\open(G),\cT_G)\cong\Sh(\open(A(P_G)),\cT_{A(P_G)})
\end{array}
$$
where we have denoted as $\coSh$ the categories of cosheaves on a given Grothendieck site.
\end{proposition}

This proposition clarifies the relation between our definitions and the
terminology in \cite{HG19},
where objects of $\Pre(P_G)$ are called cosheaves.
This also explains the homological notation we used in the case of a sheaf on a
digraph with respect to the standard topology.

\begin{observation}\label{weak-lapl-obs}
1) Let $G\in\stGraphs$
and consider the standard Grothendieck sites $(\open(A(\mathcal{P}_G)), \cT(A(\mathcal{P}_G)))$
and $(\open(G), \cT_G)$. If we choose the canonical cover of $G$ by
the open stars $\cU(G):=\lbrace U_v\rbrace_{v\in V_G}$ and we give a well-ordering to the vertices,
we have that a weak Laplacian pair with respect to $\cU(G)$ is equivalent to
the datum of two presheaves $\overline{F}\in \Pre(\mathcal{P}_G^{op})$ and
$\underline{F}\in\Pre(\mathcal{P}_G)$ together with linear maps
$f_v:\underline{F}(U_v)\rightarrow \overline{F}(U_v)$
and $f_e:\overline{F}(U_e)\rightarrow \underline{F}(U_e)$ for all $v\in V_G$, $e\in E_G$. 

\medskip\noindent
2) In general if $G\in\diGraphs$ and we consider the cover of
$G$ for the topology $\cT(A(\mathcal{P}_G))$ by the open stars $\cU(G):=\lbrace U_v\rbrace_{v\in V_G}$
with a well-order on it we note that, using the notation of Definition \ref{def:laplacianpairs},
the datum of two presheaves $\overline{F}\in \Pre(\mathcal{P}_G^{op})$
and $\underline{F}\in\Pre(\mathcal{P}_G)$
together with linear maps $f_v:\underline{F}(U_v)\rightarrow \overline{F}(U_v)$ for all $v\in V_G$
bijectively corresponds to the sheaves, the cosheaves and the morphism $F_0(\cU(G))$
involved in the definition of a weak Laplacian pair with respect to $\cU(G)$.
In order to define a weak Laplacian pair, we need
the additional datum of a linear map $f_e:\overline{F}(U_e)\rightarrow \underline{F}(U_e)$
for all $e\in E_G$,
to define a map $F_1(\cU(G))$ as in Def. \ref{def:laplacianpairs}.
Indeed, if we consider two vertices $v,w\in E_G$, we can use the datum of maps $f_e$
to build the map $\oplus f_e$

$$
\overline{F}(U_v\cap U_w)\cong\bigoplus_{e\in \lbrace v\xrightarrow{e_i}w, w\xrightarrow{l_j}v\rbrace}\overline{F}(e)
\rightarrow\bigoplus_{e\in \lbrace v\xrightarrow{e_i}w, w\xrightarrow{l_j}v\rbrace}
\underline{F}(e)\cong\underline{F}(U_v\cap U_w)
$$

Notice however, that these maps involved in the definition of weak Laplacian pair on $\cU(G)$
are not in bijective correspondence with the datum of a linear maps
$f_e:\overline{F}(e)\rightarrow \underline{F}(e)$, since not all the maps
$\overline{F}(U_v\cap U_w)\rightarrow \underline{F}(U_v\cap U_w)$ arise in this way.
\end{observation}

\subsection{Sheaf Laplacians
  on directed graphs}\label{sec:hocohlapl2}
We introduce, via Laplacian pairs, the notion of sheaf Laplacian,
which turns out to be quite important for the applications \cite{bodnar}.
As in the previous subsection, we assume all graphs to be finite and all the vector spaces we consider to be finitely generated.

Let be $G$ be in $\diGraphs$.
Consider the site
$(\open(A(P_G)),$  $\cT(A(P_G)))$, with
the canonical cover $\cU(G)$ introduced in Observation \ref{weak-lapl-obs} with a well-ordering to it.
Let $\ccF$ be a sheaf, that we identify with its restriction
to $\Pre(P_G^{op})$.

\begin{definition}\label{def:notableLaplacianpairs}
Let $\ccF$ be a (pre)sheaf
in $\Vectk$ and $\ccF^\ast$ the co(pre)sheaf defined as $\Hom_{\Vectk}(-,\bk)\circ\ccF$.
We define  the \textit{\v{C}ech Laplacian} $L_\ccF$, as the \v{C}ech Laplacian associated with the
Laplacian pair $(\ccF, \ccF^\ast, \cU(G))$ where $f_v$, $f_e$ are the standard isomorphisms
once bases are given for all $\ccF(U_v)$, $\ccF(e)$.
Alternatively, to define $L_\ccF$, 
instead of the choice of bases, we can also fix an inner product on such spaces.
\end{definition}

\begin{remark}
One can easily adapt this definition for $G \in \Graphs_{\leq 1}$.
The same is true for the rest of the definitions and propositions of this section.

\end{remark}

We give a proposition, whose proof is a simple direct check, connecting our
$L_\ccF$ to the so called sheaf Laplacian as in \cite{bodnar}.

\begin{proposition}
Let 
$L_\ccF:  C^0 (G, \ccF)\rightarrow  C^0 (G, \ccF)$ be the  \v{C}ech Laplacian
of the Laplacian pair $(\ccF, \ccF^\ast, \cU(G))$ as above. Then its explicit formula is the following:
\begin{equation}\label{sheaf-lap}
L_\ccF(x)_v = \sum_{u,v\leq e} \ccF^*_{v\leq e}(\ccF_{v\leq e}x_v-\ccF_{u\leq e}x_u)
\end{equation}
\end{proposition}

\begin{definition}\label{def:DeltaF}
    If $\ccF$ is a (pre)sheaf with values in one of the categories listed in \ref{SetstoVect} whose restriction maps are all invertible, then by inverting the restriction maps we get a cosheaf  we denote as $\ccF^{-1}\in\Pre(P_G)$. Defining now $f_v$, $f_e$ to be identity morphisms we get a strong Laplacian pair $(\ccF, \ccF^{-1}, \cU(G))$  whose \v{C}ech Laplacian we denote as $\Delta_\ccF$.
\end{definition}

\begin{observation}\label{Obs:bodnarlapl}
Let $G\in\stGraphs$ and let be $\ccF\in\Pre(\mathcal{P}_G^{op})$
be a (pre)sheaf and  $\ccF^\ast$ be either the functor $\ccF^\ast$ of Definition \ref{def:notableLaplacianpairs} or the functor $\ccF^{-1}$ of \ref{def:DeltaF}.
Then we get the (pre)sheaf $\Theta_p(\ccF)\in\Pre(P_{\Theta(G)}^{op})$ described in Observation
\ref{obs:moveshvalex2} and $\Theta_p(\ccF)^\ast\cong\Theta_p(\ccF^\ast)\in\Pre(P_{\Theta(G)})$
where $\Theta_p(\ccF)^\ast$ is the functor we would obtain starting from $\Theta_p(\ccF)$ in either Definition \ref{def:notableLaplacianpairs} or Definition \ref{def:DeltaF}.
We have that $L_{\ccF}=\frac{1}2 L_{\Theta_p(\ccF)}$ (in the case of Definition \ref{def:notableLaplacianpairs})
or $\Delta_{\ccF}=\frac{1}2\Delta_{\Theta_p(\ccF)}$ (in the case of Definition \ref{def:DeltaF}).
\end{observation}

\begin{observation}
 Consider $G\in\diGraphs$. Given a (pre)sheaf-cosheaf pairs $\overline{F}\in\Pre(P_G^{op})$, $\underline{F}\in\Pre(P_G)$, setting $\overline{F}(v\leq e)=0$ if $v$ is the head of $e$ and $\underline{F}(v\leq e)=0$ if $v$ is the tail of $e$ (or viceversa) still gives us a sheaf-cosheaf pair (no need to check functoriality). This modification gives rise to ``\textit{directed Laplacians}".
\end{observation}

\begin{remark}
Laplacian pairs can be generalised also the the context of bisheaf complexes (\ref{def:bisheafcplx}) arising from sheaves and cosheaves on a poset, thus recovering the generality of the Laplacians arising in \cite{HG19}.
As we will not need this, we shall not pursue this point of view in this work.
\end{remark}

\subsection{Cellular sheaves}\label{sec:cellsh}
For the sake of completeness, we recall here the definition of cellular sheaves (as in \cite{curry}, for example), even though we will not use it in the paper.
It is well known that a CW complex, under suitable assumptions, gives rise to a characteristic poset.
We briefly recap some basic definitions and notions.

\medskip

\begin{definition}
A \emph{regular cell complex} is a locally finite regular CW complex. In other words, a regular cell
complex is the datum of a topological space $X$ together with a a partition of $X$ into subspaces
$\lbrace X_\alpha\rbrace_{\alpha\in \textbf{P}_X}$ such that:
\begin{itemize}
\item[1)] for every $x\in X$, every sufficiently small neighbourhood of $x$ intersects only finitely many $X_\alpha$,
\item[2)] for every $\alpha,\beta\in \textbf{P}_X$, $\overline{X}_\alpha\cap X_\beta\neq\emptyset$ if and only if
$X_\beta\subseteq \overline{X}_\alpha$ (incident cells),
\item[3)] every $X_\alpha$ is homeomorphic to 
  some $\R^{n_\alpha}$,
\item[4)] for every $\alpha\in \textbf{P}_X$, there is an homeomorhism $\overline{\mathrm{D}^{n_\alpha}}\xrightarrow{\cong} \overline{X_\alpha}$ mapping the open disk $\mathrm{D}^{n_\alpha}$ homeomorphically to $X_\alpha$.
\end{itemize}
The set $\textbf{P}_X$ is a poset (order defined by the inclusion of the subspaces),
called the \emph{cell poset} of $X$.
\end{definition}

\begin{remark}
It can be proven \cite{bj} that a regular cell complex is uniquely determined by the combinatorics
of its cell poset. Moreover, any CW poset (a poset enjoying certain combinatorial
properties, see \cite{bj})
give rise to a regular cell complex.
\end{remark}

If the complex is finite, the dimension of a regular cell complex $X$ is
defined to be the dimension of its
associated poset $\textbf{P}_X$.

\begin{remark}
Regular cell complexes of dimension 1 can be identified with finite undirected graphs.
In this case, the cell poset contains all the information necessary to recover
the whole graph structure.
This is not true in the case of directed graphs, where some additional data is needed.
Indeed, a semisimplicial
set of dimension 1 is fully recovered by its category of simplices,
that has a richer structure than its
associated poset (that is unable to distinguish between heads and tails).
\end{remark}

We recall a key definition  (see \cite{curry}
and refs therein).

\begin{definition}
A \textit{cellular sheaf} $\cF$ valued in a category $\cC$ on a regular cell complex $X$
is a functor $\cF : \textbf{P}_X \lra \cC$, i.e. it is assigning to each cell $X_\alpha$ in X an object $\cF(\alpha) \in \cC$
and to every pair of incident cells $X_\al \subset  X_\be$ a restriction map
$\rho_{\alpha,\beta}: F(\alpha) \lra F(\beta)$.
A \textit{cellular cosheaf} $\widehat{\cF}$ is a functor $\cF : \textbf{P}_X^{op} \lra \cC$ assigning to each cell
$X_\alpha$ in X an object $\widehat{\cF}(\alpha)\in \cC$
and to every pair of incident cells $X_\al \subset  X_\be$ an extension map
$\epsilon_{\alpha,\beta}: F(\beta) \lra F(\alpha)$.
\end{definition}
This is in agreement with our treatment of sheaves on a poset, i.e. on
$A(\textbf{P}_X)$. We do not pursue this point of view further: see 
\cite{curry, HG19} (and refs. therein) for a more comprehensive development of such theory.

\section{Differential operators and bundles on graphs}\label{lapl-sec}
In this section we introduce differential operators on geometric structures
related to graphs. We make the blanket assumption, unless otherwise stated, that all the graphs in this section
have one self-loop for each vertex. We abuse notation and
we denote as $\diGraphs_{\leq 1}$ the category of all finite digraphs having at most
one edge for each pair of ordered vertices $x \neq y$ and exactly 
one edge, called a \textit{self loop} for each pair $(x,x)$. Similarly we denote with $\diGraphs$
the category of all finite digraphs having exactly one self loop for each vertex.

\subsection{First Order Differential Calculi and Graphs}
We first recap how to define a first order differential
calculus (FODC) on the algebra of functions on a finite set based on the work \cite{dimakis},
see \cite{bm} for more details.
We will denote as $\mathrm{(FODC)}$ the category of first order differential calculi.
We denote for brevity $V:=V_G$ and $E:=E_G$. Notice that in $\diGraphs_{\leq 1}$
we have $i:E \lra V \times V$ injective, so we identify $E$ with the
corresponding
subset of $V \times V$ and we denote an edge $(x,y)$ also as $x \to y$;
we assume $x \neq y$ unless otherwise stated.

We denote as $A:=\bk[V]$, or as $A_G$ for $G=(V,E)$, the vector space of all functions
$V\rightarrow\bk$. Note that:  
\begin{equation}\label{fin-alg}
A=\bk[V]=\mathrm{span}\{\delta_x\ |\ x\in V\},
\end{equation}
where $\de_x(y)=1$ if $x=y$ and zero
otherwise. We define a FODC $(\Gamma^1, \extd)$, 
or $(\Gamma^1_G, \extd_G)$, for $G=(V,E)$
as the $A$-bimodule whose underlying vector space is the vector space freely
generated by the set $E$, that is: 
$$
\Gamma^1:=\bk [E] = \mathrm{span}\{ \oomega_{x \to y} \, |\, (x,y) \in E\}
$$
The $A$-bimodule structure is given by:
$$
f\oomega_{x \to y}=f(x)\oomega_{x\to y},\quad \oomega_{x\to y}f=\oomega_{x\to y}f(y),\quad \extd
f=\sum_{x\to y\in E}(f(y)-f(x))\oomega_{x\to y}
$$
for all functions $f\in A$. We define $\extd: A \lra \Gamma^1$ {on generators} as:
\begin{equation}\label{d-def}
\extd\delta_x=\sum_{y: y\to x}\oomega_{y\to x}-\sum_{y:x\to y}
\oomega_{x\to y},\quad \delta_x\extd\delta_y=\begin{cases} -\sum_{z:x\to z}\oomega_{x
\to z}& x=y\\ \oomega_{x\to y}& x\to y\\ 0 &\mathrm{otherwise}\end{cases}
\end{equation}
The fact that $\extd$ is a derivation, i.e. it satisfies the Leibniz identity
is a simple check. Hence $(\Gamma^1, \extd)$ is a FODC on $A$.
This FODC is \textit{inner}, i.e. $\extd a=[\theta,a]$ for all $a\in A$, where
$$
\theta:=\sum_{x\to y\in E}\oomega_{x\to y}
$$
We shall write $\theta_{\Gamma^1}$ for $\theta$, whenever necessary to avoid
confusion.

\begin{remark}
We notice that $\theta$ is a finite analog of the Dirac operator in Connes
spectral triples in noncommutative geometry.
We shall not pursue further this point of view.
\end{remark}

We state the following result,
(see \cite[Ch. 1]{bm} and \cite{dimakis}).

\begin{theorem} \label{eqcatgr}
Let the notation be as above.
Then, we have a fully faithful contravariant
functor 
$$F: \diGraphs_{\leq 1} \lra \mathrm{(FODC)}, \qquad G=(V,E) \mapsto (\Gamma^1,\extd)$$
realizing an 
antiequivalence of categories between $\diGraphs_{\leq 1}$
and the category of FODC $(\Gamma^1, \extd)$ on $\bk$-algebras $A=\bk[V]$, with $V$ finite set,
such that the map
$A\otimes A\rightarrow \Gamma^1$, $a\otimes b \mapsto adb$ is surjective.
\end{theorem}

We briefly comment on the definition of $F$ on morphisms and the proof of Thm \ref{eqcatgr}.
Any arrow $\phi:G\lra H$ in $\diGraphs_{\leq 1}$
will induce the following map: {$\phi^*:A_H \lra A_G$, $\phi^*(a)(v)=a(\phi_{V}(v))$} and 
\begin{equation}\label{phi*}
\phi_*:\Gamma^1_H \to \Gamma^1_G,\quad
\phi_*(\oomega_{x\to y})=\sum_{ w\to z \atop  w\in \phi_{V}^{-1}(x),\ z\in\phi_V^{-1}(y)}\oomega_{w\to z}
\end{equation}
where $\phi_V:V_G\lra V_H$ and $\phi$
is a pair, $\phi=(\phi_V:V_G \lra V_H,\phi_E:E_G \lra E_H)$.
One can check that $\phi^*$ is a bimodule map,
moreover $\phi_*(\extd a)=\extd\phi^*(a)$,
$a \in A_H$ i.e. this
is an arrow in (FODC).
Since both calculi are inner to prove this it suffices to show that
$\phi_*(\theta_{\Gamma^1_G})=\theta_{\Gamma^1_H}$.

\begin{definition}\label{def:fullyconnected} Given a finite set $V$,
  we define the \textit{fully connected} or
\textit{complete} directed graph {$G \in \diGraphs_{\leq 1}$} as the graph having for vertices the set $V$ and as edges
one edge in each direction for every pair of vertices. Finally, as always,
we have for each vertex exactly one self loop.
\end{definition}
\begin{observation}\label{obs:univcalculus}
The equivalence of categories in {Theorem \ref{eqcatgr} prescribes} that we have a FODC associated with the fully connected graph having $V$ as set of vertices. We shall call the resulting FODC the \textit{universal FODC} and denote as $\Omega_V^1$. Any FODC on $A=\bk[V]$ is a quotient of the universal FODC (see \cite{bm} for more details).
\end{observation}

\subsection{First order differential calculi and \'etale coverings of graphs}
Let us now generalize the treatment, for applications in
the next sections.

\begin{definition}
Let $G \in \diGraphs$. We say that the \'etale cover $\phi:H \lra G$ is an
\textit{\'etale directed cover} if
\begin{enumerate}
\item $H$ is a disjoint union of graphs in $\diGraphs_{\leq 1}$.
\item The arrow $\phi_E:E_H \lra E_G$ induced by $\phi$ is bijective when restricted to non self-loops.
\end{enumerate}
Notice the second part of condition (2) is necessary since $H$, as $G$, has one self loop at each vertex.
Clearly, given $G$, such $\phi$ is not unique, however it is immediate
to check that one always exists.
\end{definition}

\begin{definition}\label{fodc-etale}
Let $\phi:H\lra G$ be an \'etale directed cover of $G\in \diGraphs$,
$H=\coprod H_i$, $H_i \in \Graphs_{\leq 1}$. We
define $(\Gamma^1, \extd)$ of the pair $(G,\phi)$ as follows.

Let $(\Gamma^1_i,d_i)$ be the FODC associated to $H_i$,
we define $\Gamma^1:=\oplus \Gamma^1_i$ and 
$$
\extd:\bk[V_G] \lra \Gamma^1
\qquad \extd a:=\sum_i \extd_i(\phi^*(a)_{|V_{H_i}})
$$
where  $\phi^*:\bk[G] \lra \bk[H]$ is induced by $\phi$ as above.
\end{definition}
\begin{proposition}\label{fodc-gen}
Let $\phi:H\lra G$ be an \'etale directed cover of $G\in \diGraphs$ and $(\Gamma^1, \extd)$
as above. Then $(\Gamma^1, \extd)$ is a FODC on $A_G:=\bk[V_G]$, with bimodule structure:
$$
a \cdot \oomega_{x \to y} \cdot b := a(\phi(x))b(\phi(y))\oomega_{x \to y}
$$
Moreover this FODC is independent from the chosen \'etale direct cover.
\end{proposition}

\begin{proof} The fact  $(\Gamma^1, \extd)$ is a FODC on $A_G$ is a direct check.
For the last statement, let $(\Gamma^{1'},\extd')$ be a FODC on $A_G$ obtained
with a different \'etale direct cover $\phi':H' \to G$. Notice that, by the very definition
of \'etale directed cover,
$$
|\phi_E^{-1}(E_G)|=|(\phi')_E^{-1}(E_G)|
$$
Hence we obtain a linear isomorphism $\psi:\Gamma^1 \to \Gamma^{1'}$ that one can
check to be also a bimodule isomorphism. We see immediately also that:
$$
\extd'(a)=\psi \extd(a)
$$
\end{proof}

\begin{example}
Consider the graph $G$ and its \'etale covering 
$$f:H=H_1 \amalg H_2 \amalg H_3 \longrightarrow G$$
with $G,H \in \diGraphs$ (self-loops are not depicted, but we have exactly one
at each vertex):
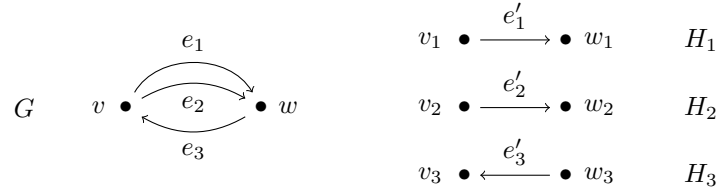
\begin{figure}[h]
\begin{center}
\begin{tikzpicture}[scale=.45]

\node at (-5,0) {$G$};
\node (v) at (-2,0) {$\bullet$};
\node (w) at (2,0) {$\bullet$};
\node at (-2.8,0) {$v$};
\node at (2.8,0) {$w$};
\node at (0,1.9) {$e_1$};
\node at (0,0.1) {$e_2$};
\node at (0,-1.3) {$e_3$};

\draw [->] (w) to [bend left] (v);
\draw [->] (v) to [bend left] (w);
\draw [->,out=60,in=120] (v) to (w);

\node at (15,2) {$H_1$};
\node at (15,0) {$H_2$};
\node at (15,-2) {$H_3$};
\node (v1) at (8,2) {$\bullet$};
\node (w1) at (11,2) {$\bullet$};
\node (v2) at (8,0) {$\bullet$};
\node (w2) at (11,0) {$\bullet$};
\node (v3) at (8,-2) {$\bullet$};
\node (w3) at (11,-2) {$\bullet$};
\node at (7,2) {$v_1$};
\node at (12,2) {$w_1$};
\node at (7,0) {$v_2$};
\node at (12,0) {$w_2$};
\node at (7,-2) {$v_3$};
\node at (12,-2) {$w_3$};
\node at (9.5,-1.3) {$e_3'$};
\node at (9.5,0.7) {$e_2'$};
\node at (9.5,2.7) {$e_1'$};

\draw [->] (v1) edge (w1);
\draw [->] (v2) edge (w2);
\draw [->] (w3) edge (v3);

\end{tikzpicture}
\end{center}
\end{figure}

We have that
$\Gamma_G^1:=\mathrm{span}\{\omega_{v_1 \to w_1}, \omega_{v_2 \to w_2}, \omega_{w_3 \to v_3}\}$,
and
$$
\extd(a)=\extd_1(f^*(a)_{|V_{H_1}})+\extd_2(f^*(a)_{|V_{H_2}})+\extd_3(f^*(a)_{|V_{H_3}}) \in \Gamma^1=\Gamma^1_{H_1}\oplus\Gamma^1_{H_2}\oplus\Gamma^1_{H_3}
$$
So, for example, if $a=\de_v$, $\extd(\de_v)=\extd_1(\de_{v_1})+\extd_2(\de_{v_2})+\extd_3(\de_{v_3})$.
\end{example}

\begin{definition}\label{fodc-gen-def}
Let $G=(V,E) \in \diGraphs$. We call $(\Gamma^1_G,\extd)$ as in Prop. \ref{fodc-gen}
a FODC on $A_G$ associated to $G$.
\end{definition}

Notice that this FODC is inner, with $\theta=\sum_{u \to v \in E_H} \oomega_{u\to v}$
for $H=(V_H, E_H)\in \diGraphs$
as in Prop. \ref{fodc-gen}.

\begin{theorem}\label{thm:digraphfodc}
The contravariant functor
$$
F: \diGraphs \lra \mathrm{(FODC)_e}, \quad G  \mapsto (\Gamma^1_G,\extd) 
$$
is {faithful and essentially surjective},
where $\mathrm{(FODC)_e}$ consists of all the FODC $(\Gamma^1, \extd)$ 
such that
\begin{itemize}
\item $A=\bk[V]$ with $V$ any finite set,
\item $\extd:A \lra \Gamma^1$ is the composition $\extd=\extd'\circ\varphi$ where
$\extd'$ is a FODC isomorphic to a FODC $(d':\bk[V']\rightarrow{\Gamma^1}')$
in the essential image of the functor considered in Theorem \ref{eqcatgr} and $\varphi:A\rightarrow\bk[V']$
is a $\bk$-algebra morphism.
\end{itemize}

\end{theorem}

\begin{proof}
The proof
is a check that we leave to the reader.
\end{proof}

\begin{observation}
Let $f:G\rightarrow G'$ be a morphism between digraphs. Let $g=(g_A, g_{\Gamma^1})$ be its image under the functor $F$ of Theorem \ref{thm:digraphfodc}, where $g_A: A_{G'}\rightarrow A_{G}$ and $g_{\Gamma^1}:\Gamma^1_{G'}\rightarrow\Gamma^1_G$. Then $g_{\Gamma^1}$ satisfies the following condition: its associated matrix with respect to the choice of the bases on $\Gamma^1_{G'}$ and $\Gamma^1_G$ given by the edges of $G'$ and $G$ (see Def. \ref{fodc-etale}) has entries only equal to $0$ or $1$ and
each row has just one non zero entry.
This property characterizes the morphisms in the image of $F$.
\end{observation}

\subsection{Quantum Differential Calculi and Graphs}\label{2-forms}
In this section we make observations
on higher order differential calculi and graphs
that we shall need in the sequel  (see \cite{dimakis, bm}).

\begin{definition}
Let $G=(V,E) \in \diGraphs_{\leq 1}$.
We define a \textit{triangular clique} as an ordered set of 3 vertices $(x,y,z)$ of $V$, 
denoted as $\omega_{x\to y \to z}$, such that $x\to y$, $y\to z$ and $x\to z$ are edges of $G$ and $x \neq y$, $y \neq z$.
\end{definition}
Notice that we allow $x=z$, so that a triangular clique in this case amounts to
two ordered
vertices $x\neq y$ satisfying the condition of the previous definition.

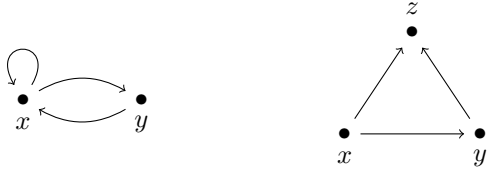
\begin{figure}[h!]
\begin{center}
    \begin{tikzpicture}[scale=.45]
    \node (x) at (-6.5,1) {$\bullet$};
    \node (y) at (-3,1) {$\bullet$};
    \node at (-6.5,0.3) {$x$};
    \node at (-3,0.3) {$y$};
    \draw [->] (x) to [out=30,in=150] (y);
    \draw [->] (y) to [out=-150,in=-30] (x);
    \draw [->] (x) to [out=60,in=120, loop] (x);

    \node (x) at (3,0) {$\bullet$};
    \node (y) at (7,0) {$\bullet$};
    \node (z) at (5,3) {$\bullet$};
    \node at (3,-0.7) {$x$};
    \node at (7,-0.7) {$y$};
    \node at (5,3.7) {$z$};
    \draw [->] (x) to  (y);
    \draw [->] (y) to  (z);
    \draw [->] (x) to  (z);
    \end{tikzpicture}
    \end{center}
    \caption{{The triangular cliques $\omega_{x\to y \to z}$ (right) and $\omega_{x\to y \to x}$ (left).}}\label{tr-cl}
\end{figure}

For a given $G \in \diGraphs_{\leq 1}$, we define $\Gamma_G^2$ as the vector space generated by the triangular cliques
\begin{equation}\label{2forms}
\Gamma_G^2:=\mathrm{span}\{ \oomega_{x \to y \to z} \, | \, x\to y, y \to z ,x\to z\in E, \, x\neq y, y\neq z\}
\end{equation}
$\Gamma_G^2$ is an $A_G$-bimodule:
$$
f\oomega_{x \to y \to z}=f(x)\oomega_{x\to y \to z},\quad \oomega_{x\to y \to z}f=\oomega_{x\to y \to z}f(z), \quad
f\in A_G
$$
In analogy to the continuous setting, we refer to $\Gamma^2_G$ as the space of 2-forms on $A_G$.
\par
We now turn to the special case of $G$ fully connected.
Let $\Omega_V^1$ be the universal FODC as defined in Observation \ref{obs:univcalculus}
and define $\Omega_V^2$ to be the $\bk$-vector space freely generated by the triangular cliques of $G$:
$$
\Omega_V^2:=\mathrm{span}\{ \oomega_{x \to y \to z} \, | \, x,y,z \in V, \, x\neq y, y\neq z\}
$$
We define the \textit{exterior product} as the $\bk$-linear map:
$$
\Omega^1_V \times \Omega^1_V \lra \Omega^2_V, \quad  (\oomega_{x \to y},  \oomega_{w \to z})\mapsto
\oomega_{x \to y}\wedge  \oomega_{w \to z}:=\de_{y,w}\oomega_{x \to y \to z}
$$
where $(\Omega^1_V,\extd_V^0)$ is the FODC associated with
the fully connected graph $G$. 
The exterior product allows us to define the \textit{differential}
as the $\bk$-linear map:
\begin{equation}\label{diff-def2}
\extd_V^1: \Omega^1_V \lra \Omega^2_V \qquad \extd_V^1  \oomega_{x \to y} :=\extd_V^0 \de_x \wedge \extd_V^0 \de_y
\end{equation}

We have the following result, whose proof is left to the reader (see \cite{dimakis}).

\begin{proposition}
The map $\extd_V^1$ satisfies the Leibniz rule:
$$
\extd_V^1(f\omega_{x\lra y})=\extd_V^0 f \wedge \omega_{x\lra y} + f \extd_V^1 \omega_{x\lra y}, \,
\extd_V^1(\omega_{x\lra y}f)=\extd_V^1 \omega_{x\lra y}f - \omega_{x\lra y} \wedge \extd_V^0 f
$$
and $\extd_V^1 \circ \extd_V^0=0$.
We also have the explicit expression: 
\begin{equation}\label{expr-d}
\extd\omega_{x\to y}=\sum_{u \in V} \left(\omega_{u\to x\to y}-\omega_{x\to u\to y}+\omega_{x\to y\to u}\right)
\end{equation}
\end{proposition}

We now show how the notion of differential calculus
can be extended beyond the first order.
Let $G=(V,E) \in \diGraphs_{\leq 1}$, 
we can write the FODC $\Gamma^1_G$ as a quotient:
$$
\Gamma^1_G=\Omega^1_V/I, \qquad I=\mathrm{span}\{\omega_{x \to y} \,|\, x \to y \not\in E\}
$$
(see \cite{majidpaper, bm}). 
The previous proposition, along with the definition of the wedge product,
allows us to see $\Omega^\bullet_V:=\oplus_{i=0}^2\Omega^i_V$ (where $\Omega^0_V:=\bk[V]$) as a differential graded algebra $(\Omega^\bullet_V,d_V^\bullet)$
(DGA, see \cite{dimakis} for a definition in this context).

\begin{proposition}
Let $G\in \diGraphs_{\leq1}$. The graded $A_G$-bimodule $\Gamma^\bullet_G:=\oplus_{i=0}^2\Gamma^i_G$, where
$$
\Gamma^0_G:=A_G, \qquad
\Gamma^1_G=\Omega^1_V/I\qquad \Gamma^2_G\cong\Omega^2_V/\extd_V^1(I)
$$
has a well defined
DGA structure induced by the one of $(\Omega^\bullet_V,\extd_V^\bullet)$, the bimodule structure being the same.
\end{proposition}

\begin{proof} Direct check.\end{proof}
We notice that any quotient of $\Gamma^2_G$ by the span of a subset of the triangular cliques will
give a well defined differential according to the formula (\ref{diff-def2}).
This prompts the following definition.

\begin{definition}\label{def:2ordcalc}
Let $G=(V,E) \in \diGraphs_{\leq1}$ and $S$ a subset of its triangular cliques.
We define the pair  $(\Gamma^\bullet_S,d_S^\bullet)$ with:
\begin{equation}
\Gamma^\bullet_S:=\Gamma^\bullet_G/\langle S \rangle, \qquad \extd_S^\bullet:\Gamma^\bullet_S \lra \Gamma_S^\bullet
\end{equation}
a \textit{second order differential calculus} on $A=\bk[V]$, where
$\langle S \rangle$ is the $A_G$-bimodule generated by $S$ and 
$\extd_S^1$
is obtained from $\extd_G^1$, by taking
the quotient of $\Gamma^2_G$ by $\langle S \rangle$.
\end{definition}

Note that if $S=\emptyset$ we get that $(\Gamma^\bullet_S,d_S^\bullet)=(\Gamma^\bullet_G,d_G^\bullet)$. In addition, notice that $\extd_S^1$
satisfies the Leibnitz rule and $\extd_S^1 \circ \extd_S^0=0$, where $\extd_S^0=\extd_G^0$ and $\Gamma^i_S:=\Gamma^i_G$
for $i=0,1$.

\begin{remark}
Let $V$ be a finite set. Our approach could
be extended to obtain all differential graded algebras on $A$ as quotients
of the universal one $\Omega_V:=\oplus_n \Omega_V^n$,
where $\Omega_V^n$ is constructed similarly to (\ref{2forms}), see \cite{dimakis}.
Moreover one could also extend our results to comprehend the case of \'etale directed covers, however this is not straightforward and we shall not pursue this point of a view furtherly since it would lead us outside the scope of the present paper.
\end{remark}

\subsection{Connections on Graphs} \label{conn-vb-sec}
We start by defining vector bundles on graphs
and connections on graphs generalizing \cite{gaybalmaz, bm, dimakis}.
We see objects in $\Graphs$ as objects in $\diGraphs$ 
via the functor described in Proposition \ref{prop-graphffemb},
where an unoriented edge is replaced by two oriented
edges, one for each direction.
Notice that such functor is fully faithful when restricted to $\Graphs_{\leq 1}$.
However in general for $\Graphs$ is still faithful, but it may not be full. Recall that in this section we are considering graphs having exactly one self loop per vertex, and that the functor $\Theta$ defined in Section \ref{sec:undgr} preserves this property.

\begin{definition}
    We say that a digraph $G\in\diGraphs$ is \textit{bidirected} if it is of the form $\Theta(H)$ for some $H\in\Graphs$.
\end{definition}

\begin{definition}
A \textit{vector bundle} $\cF$ of rank $n$ on a {set $V$}
is an assignment:
$$
v \lra \cF_u, \qquad v \in V
$$
where $\cF_v$ is a vector space of dimension $n$.
We define the \textit{frame bundle} $\mathrm{Fr}$, an assignment:
$$
V \ni v \mapsto \{e^v_i\} \subset \cF_v
$$
where $\{e^v_i\}$ is a basis for $\cF_v$. Moreover we denote with
$\id_{u,v}: \cF_u \lra \cF_v$ the linear map $\id_{u,v}(e_i^u)=e_i^v$.
\end{definition}

\begin{definition}\label{GB-parallel}
Let $\cF$ be a vector bundle on $V$ and let $G=(V,E) \in \diGraphs$.
We define a \textit{weak parallel transport} 
a collection of linear maps $\mathcal{R}_{e,u \to v}:\cF_v \lra \cF_u$,
where $e$ is an edge between $u$ and $v$.
For each (unique) self loop $u\rightarrow u$, we impose that $\cR_{u\to u}=\id_{\cF_u}$.
If $G \in \diGraphs_{\leq 1}$ is bidirected, we say that a weak parallel transport is a
\textit{parallel transport} if each $\mathcal{R}_{e,u \to v}$
is invertible and $\mathcal{R}_{e,u \to v}=\mathcal{R}_{e',v \to u}^{-1}$, where
$e'$ is the unique edge from $v$ to $u$ via the functor in Prop. \ref{prop-graphffemb}.
To ease the notation, when $G \in \Graphs_{\leq 1}$, we omit $e$ and write  $\mathcal{R}_{u \to v}$.
We define a \textit{connection} on a digraph $G$
as a collection of linear maps
$\Theta_{e, u \to v}:= \mathcal{R}_{e, u \to v} - {\id_{v,u}}$,
on all edges $e \in E$, with $\{\mathcal{R}_{e,u \to v}\}$ a weak parallel transport.

Once a frame bundle is given, we can write (implicit summation on repeated indices):
$$
\mathcal{R}_{e,u \to v}:\cF_v \lra \cF_u, \qquad e_i^v \mapsto \mathcal{R}_{e, u\to v, i}^j e_j^v
$$
So we can view a connection as
an element of $\Gamma^1_G\otimes_{A_G}\mathrm{M}_n(A_G)\cong\mathrm{M}_n(\Gamma^1_G)$
in line with \cite{dimakis}.
\end{definition}

\begin{observation}\label{obs-eur}
In the differentiable setting the parallel transport for a vector bundle $E \lra M$ on
a differentiable manifold $M$ is a collection of maps:
$$
{\displaystyle \Gamma (\gamma )_{s}^{t}:E_{\gamma (s)}\rightarrow E_{\gamma (t)}}
$$
It allows us to take the derivative of a section
$V$ along a curve $\gamma$: 
\begin{equation}\label{ord-par-tr}
\nabla _{\dot \gamma} V =\lim _{h\to 0} \frac {\Gamma (\gamma )_{h}^{0}V_{\gamma (h)}-V_{\gamma (0)}}{h}
=\left.{\frac {d}{dt}}\Gamma (\gamma )_{t}^{0}V_{\gamma (t)}\right|_{t=0}.
\end{equation}
We now euristically rewrite the above formula, replacing the curve $\gamma$ with
an edge $e$ between vertices $u$ and $v$ (taking the places of $\gamma(0)$ and $\gamma(h)$) of
the graph $G \in \Graphs$. The vector bundle $E$
is replaced by $\cF$, so that $V_u \in \cF_u$ and we set $h=1$ in the limit. We obtain 
\begin{equation}\label{eq-mgb}
V_u \mapsto \mathcal{R}_{e,u\to v} V_v - V_u 
\end{equation}
In this euristic identification, $\nabla$ is replaced with a family of 
$\mathcal{R}_{e, u\to v}-\mathrm{id}$ for each edge $e$ between $u$ and $v$ and for all pair
of vertices $(u,v)$ linked by an edge,
in agreement with Def. \ref{GB-parallel} of connection for vector bundles on graphs.
\end{observation}

\subsection{Noncommutative Connections and vector bundles}
We now want to take a noncommutative, though equivalent, point of view on connections,
generalizing \cite[Ch. 3, 3.2]{bm}, beyond $\Graphs_{\leq 1}$,
that will allow us to write a connection $\Theta=(\Theta_{e, u \to v})$
in a more effective way.

\begin{definition}\label{maj-conn}
Let $G=(V,E) \in \diGraphs$, $A=\bk[V]$ and $(\Gamma^1, d)$ the FODC on $A$ as in Def.
\ref{fodc-etale}. Let $M$
be a free rank $n$ left $A$-module. We define a \textit{left noncommutative connection} $\nabla$ on $M$ as a map $\nabla:M \lra \Gamma^1 \otimes M$
satisfying the Leibniz identity, i.e:
$$
\nabla(fm)=\extd f \otimes m + f \nabla m,\qquad f \in A, \quad m \in M
$$
Analogously, given a free rank $n$ right $A$-module $M$, one can define a \textit{right noncommutative connection} $\nabla$ on $M$ as a map 
$\nabla:M \lra  M \otimes \Gamma^1$
satisfying the Leibniz identity:
$$
\nabla(mf)=m \otimes \extd f + (\nabla m)f, \qquad f \in A, \quad m \in M
$$

\end{definition}
Once a basis $\{ e_i\}_{i=1}^n$ for the free $A$-module $M$ is chosen, a non commutative right
connection amounts to give a map:
$$
e_if^i\mapsto e_i \otimes df^i + 
e_j\otimes\oomega^j_if^i
$$
where $\oomega^j_i$ is a matrix of 1 forms, i.e. elements of $\Gamma^1$.

\begin{observation}\label{GB-Majid-equiv}
Let the notation be as above.
Let $G=(V,E)  \in \diGraphs$ and $A= \bk[V]$. There is a bijective correspondence between the
two notions:
\begin{enumerate}
\item A noncommutative right connection on $M$,
a right $A$-module of rank $n$, with respect to the FODC given via $G$ on $A$.
\item A connection on a digraph (or equivalently a weak parallel transport see Def. \ref{GB-parallel}).
\end{enumerate}
(2) $\rightarrow$ (1). In fact, consider a vector bundle $\cF$ of rank $n$ on $V$,
a frame bundle
$$V \ni v \mapsto \{e^v_i\}$$
and a free rank $n$ right $A$-module $n$ with the choice of a basis $\{ e_i\}_{i=1}^n$ (note that $\{ \delta_ve_i\}_{i,v}$ will then be a basis for $M$ as a $\bk$-vector space). Then given a connection $\Theta_{e, u \to v}:= \mathcal{R}_{e, u \to v} - {\id_{v,u}}$ by setting  
\begin{equation}\label{eq:omega}
\oomega^j_i = \sum_{e, x \to y}[\mathcal{R}_{{e, x \to y, i}}^j -\de_{i,j}]\oomega_{e, x\to y}
\end{equation}
we get a noncommutative right connection on $M$. Explicitly note that the connection given by the assignment (\ref{eq:omega}) is determined by the following formula
\begin{equation}\label{conn-f}
M \ni e^if_i \mapsto 
\sum_{e, x \to y}
e_j\otimes\left[ f^i(y) \mathcal{R}_{{e, x \to y, i}}^j  -
f^i(x)\delta_{ij}\right]\oomega_{e, x \to y} \in M \otimes \Gamma^1_G
\end{equation}this is reminiscent of the formula (\ref{ord-par-tr}) found in the classical theory of parallel transport for a vector bundle on a manifold.

(1) $\rightarrow$ (2). Conversely given a right connection
$$
e_if^i\mapsto e_i \otimes df^i + e_j\otimes\oomega^j_if^i
$$ where $\omega^j_i=\sum_{e, x \to y}a^j_{e, x \to y,i}\omega_{e, x \to y}$ (using the basis $\{\omega_{e, x \to y}\}$ of $\Gamma^1_G$ as a $\bk$-vector space), by setting $R_{e, x \to y,i}^j:=a^j_{e, x \to y,i}+\delta_{ij}$ we get a connection.
\end{observation}

\subsection{Curvature}
We now define the {\sl curvature} of a connection.
Recall that in Sec. \ref{2-forms}, given $G \in \diGraphs$
and a subset $S$ of its triangular cliques, we have given the notion of 2-forms on $G$,
as elements of a second order differential calculus (see Def. \ref{def:2ordcalc}).

\begin{definition}\label{def:curvature}
Let $G \in \diGraphs_{\leq 1}$, $S$ a subset of its triangular cliques and $M$ a free $A$
bimodule of rank $n$ having as a basis $\{e_i\}_{i=1}^n$.
Given a connection $\nabla:M\rightarrow \Gamma^1_G\otimes M$ as above 
we define:
\begin{itemize}
\item its \emph{curvature} $R_\nabla:M\rightarrow M\otimes\Gamma^2_G$
as the right $A$ module map defined on the basis $\{e_i\}_{i=1}^n$ as follows
$$
R_\nabla(e_i)=
e_j\otimes d\oomega^j_i+
e_j\otimes 
\oomega^j_k\wedge\oomega^k_i
$$ 
\item its \emph{curvature outside of} $S$ as:
$$R_\nabla^S:=(\id\otimes\pi_S)\circ R_\nabla:M\rightarrow M\otimes \Gamma^2_S$$
where $\pi_S:\Gamma^2_G\rightarrow\Gamma^2_S=\Gamma^2_G/\langle S \rangle$ is the projection morphism.
\end{itemize}
We say that $\nabla$ is \emph{flat outside of $S$} if $R_\nabla^S=0$.
We say that $\nabla$ is \emph{flat} if $R_\nabla=0$.
\end{definition}
Notice that when $G$ is the fully connected digraph in $\diGraphs_{\leq 1}$ and $S=\emptyset$ we recover
the definition of curvature as in \cite{dimakis} and Definition 3.18 of \cite{bm}.
Using equation (\ref{eq:omega}) we can rewrite $R_\nabla$ in terms the weak parallel transport associated with $\nabla$ as

\begin{equation}\label{eq:curv}
  R_\nabla(e_i)=\sum_{x\to y\to z\in \mathrm{tri}(G)}
  (\mathcal{R}_{x\rightarrow y,k}^j\mathcal{R}_{y\rightarrow z,i}^k-\mathcal{R}_{x\rightarrow z,i}^j)e_j\otimes\omega_{x\to y\to z}
\end{equation}
where we have omitted the edge $e$ in $\mathcal{R}_{e, x \to y} $ because of our assumptions on $G$ and we have denoted as $\mathrm{tri}(G)$ the set of all triangular cliques of $G$.
 Notice that in (\ref{eq:curv}) we are not summing on the index $i$.
Then, by equation (\ref{eq:curv}) we get the following result.

\begin{proposition}
Let be $G$, $M$ and $\nabla$ as above. Then:

1. If $\nabla$ is flat then $\mathcal{R}_{x\rightarrow z}=\cR_{x\rightarrow y}\mathcal{R}_{y\rightarrow z}$ for each triangular clique. In particular, we have that $\cR_{x\to y}=\cR_{y\to x}^{-1}$ for all edges $x\to y\in E_G$ that are part of a triangular clique of the form $x\to y\to x$.  

2. Assume $G$ to be bidirected.
Consider the set of triangular cliques $S$ consisting of all triangular cliques of the form $x\to y\to z$ having $x,y,z\in V_G$ all distinct.
Then $\nabla$ is flat outside of $S$ if and only if $\cR_{x\to y}=\cR_{y\to x}^{-1}$ for all edges $x\to y\in E_G$, that is the weak parallel transport associated to $\nabla$ is a parallel transport. 

\end{proposition}
Notice that if $\nabla$ is flat and assuming that $G$ is bidirected, by case 1. of the previous proposition, the weak parallel transport associated to $\nabla$ is a parallel transport having curvature zero in the sense of \cite[Section 3.4]{gaybalmaz}. If $S$ is a different set of triangular cliques, then the notion of flatness outside of it can encode different meanings. To the extreme, if $S=\mathrm{tri}(G)$, as $\Gamma^2_S$ becomes trivial, the condition of flatness outside of $S$ becomes empty and we do not get any additional property or relation on the weak parallel transport associated with $\nabla$.

\medskip
Let $G$ be either an element of $\diGraphs$ or $\Graphs$ and let $\underline{\bk}$ be the constant sheaf on $(A(P_G), \cT_{A(P_G)})$
 with value $\bk$ in the notation of
 Sec. \ref{section:sheavesposet}.

\begin{definition}
Let be $\mathcal{V}$ a sheaf on $(A(P_G), \cT_{A(P_G)})$, where $G$ can be either an element of $\diGraphs$ or $\Graphs$. We say that $\cV$ is a $\underline{\bk}$\textit{-vector bundle} of rank
$n$ on $G$ if it is a locally free
sheaf of $\underline{\bk}$-modules of rank $n$ on $(A(P_G), \cT_{A(P_G)})$
(see also \cite[\href{https://stacks.math.columbia.edu/tag/01C6}{Tag 01C6}]{SP}).
\end{definition}

\begin{observation}
Unraveling the definitions, because of \ref{bod-sh} we get that a $\underline{\bk}$-vector bundle $\cV$
of rank $n$ on $G\in\diGraphs$ amounts to the datum of a presheaf $V\in\Pre(P_G^{op})$,
such that $V(U_a)\cong \bk^n$,
$\forall a\in P_G$ and
$V_{v\leq e}:V(U_v) \lra V(U_e)$ is an isomorphism for all $v\leq e$.
\end{observation}

\begin{definition}\label{def:partranspvect}
Let be $V$ a $\underline{\bk}$-vector bundle of rank $n$ on $G\in\diGraphs_{\leq 1}$. We define
the \textit{weak parallel transport associated to} $V$ as $\lbrace\cR^V_{e,u\to v}:=V_{u\leq e}^{-1}V_{v\leq e}\rbrace_{e\in E_G}$. If $V$ is a sheaf of inner product spaces (or Hilbert spaces) of dimension $n$ and we have a frame bundle on $V$ seen as a vector bundle on $V_G$ we define the parallel transport associated to $V$ as $\lbrace\cR^V_{e,u\to v}:=V_{u\leq e}^{\ast}V_{v\leq e}\rbrace_{e\in E_G}$ where we have denoted as $V_{u\leq e}^{\ast}$ the adjoint of $V_{u\leq e}$.
\end{definition}

\begin{remark}
    We explicitly note that, if we consider a differentiable (or holomorphic) manifold $M$, there is a bijective correspondence between vector bundles with a flat connection on it and isomorphisms classes of locally constant sheaves of real (complex) vector spaces on it (see \cite{Voisin} Section 9.2). For a vector bundle with a flat connection $(V,\nabla)$ on $M$, given a trivializing ordered cover $\cU$ for the locally constant sheaf $L_V$ associated to  $(V,\nabla)$ by this correspondence, we get from Observation 4.36 in \cite{FSZ26} (see \emph{op. cit.} for the notation and the terminology) a finite ringed space $(M_\cU^2, \cO_{M_\cU^2})$ and a sheaf of $\mathbb{R}$-vector spaces $V_{M_\cU^2}$ on $M_\cU^2$.  The latter gives rise to a flat connection on $\Gamma^1_{G(\cU)_\bullet}$ where $G(\cU)$ is the graph having one vertex for each element of $\cU$ and an edge between two vertices if the intersection of the corresponding elements of $\cU$ is not empty. As a consequence, our definition of flat connection does not only mimics the one arising in the case of manifolds but, under certain assumptions, it actually descends from it.
\end{remark}

\subsection{Laplacians in the category of directed graphs} 
We define the Laplacian operator in the category $\diGraphs$ \cite{dimakis, bm} and
we compare it with the Laplacian as discussed for standard graphs in \cite{diestel, godsil}
and in machine learning applications(\cite{bronstein, fz} and refs. therein).

\medskip
We first extend to the category $\diGraphs$ some key notions.

\begin{definition}\label{graph-lapl}
Let $(\Gamma^1_G,\extd_G)$ be a FODC on $\bk[V]$ associated to
$G=(V,E) \in \diGraphs$. 
We define a \textit{generalized quantum metric} on $\Gamma^1$,
\textit{metric} for short, a bimodule map
$(,):\Gamma^1_G \otimes_{A_G} \Gamma^1_G \lra A_G$.
We say that a $\bk$-linear map $\Delta:A_G\to A_G$ is a \textit{second order Laplacian} if
$$
\Delta(ab)=(\Delta a)b+a\Delta b+2(\extd a,\extd b)
$$
We define two notable \textit{graph laplacians} associated to the metric $(,)$ (see \cite{bm} Definition 1.17):
$$
\Delta_\theta(a):=2(\theta,\extd a)\quad,\quad_\theta\Delta(a):=-2(\extd a,\theta)
$$
where $\theta$ as in Def. \ref{fodc-gen-def}.
\end{definition}

\begin{proposition}\label{met-expr}
Let $(\Gamma^1,\extd)$ be a FODC on $G \in \diGraphs$.
Any metric on $\Gamma^1$ is expressed as:
\begin{equation}\label{rem-met}
(\oomega_{x\to y}, \oomega_{z \to t})=\de_{\phi(x),\phi(t)}\de_{\phi(y),\phi(z)}\lambda_{x\to y,z\to t}\de_{\phi(x)}
\qquad x,y,z,t \in V_H
\end{equation}
where $\de_{u,v}$ as usual denotes the Kronecker delta
and  $\phi:H \lra G$ an \'etale direct covering to define $(\Gamma^1, \extd)$.
\end{proposition}

\begin{proof} It is a direct application of the definition of bimodule map:
$$
\begin{array}{c}
(\oomega_{x\to y} \cdot a, \oomega_{z \to t})=(\oomega_{x\to y}, a\cdot \oomega_{z \to t})
\quad \Rightarrow \quad (a(\phi(y))-a(\phi(z))(\oomega_{x\to y}, \oomega_{z \to t})=0\\ \\
(a\cdot\oomega_{x\to y}, \oomega_{z \to t})=(\oomega_{x\to y}, \oomega_{z \to t}\cdot a)
\quad \Rightarrow \quad (a(\phi(x))-a(\phi(t))(\oomega_{x\to y}, \oomega_{z \to t})=0
\end{array}
$$
The fact $a(\oomega_{x\to y}, \oomega_{z \to t})=a(\phi(x))(\oomega_{x\to y}, \oomega_{z \to t})$
implies that $(\oomega_{x\to y}, \oomega_{z \to t})$ $=$ \break
$\lambda_{x\to y,z\to t}\de_{\phi(x)}$ for a scalar
$\lambda_{x\to y,z\to t}\in \bk$.
\end{proof}

\begin{remark} 

As a consequence of the previous proposition, if $G\in\diGraphs_{\leq 1}$,
a metric is uniquely determined by scalars

$\lambda_{x\to y\to x}=(\oomega_{x\to y},\oomega_{y\to x})$

where $x\to y\to x$ are triangular cliques (compare with \cite{majidpaper} 3.4).  
\end{remark}
This proposition allows us to write a more explicit formula
for the Laplacian as in Def. \ref{graph-lapl}.
\begin{proposition}
Let the notation be as above. Let $G \in \diGraphs$, $(,)$ a metric and
$\phi:H\rightarrow G$ an \'etale {directed} cover. 
Then the graph Laplacian $\Delta_\theta:A_G\lra A_G$
associated to the metric $(,)$ is given by:
$$
(\Delta_\theta a)(v)=2\sum_{x \to y, z\to t\in E_H, \phi(x)=v}
\de_{\phi(x),\phi(t)}\de_{\phi(y),\phi(z)}\lambda_{x\to y,z\to t}
(a(\phi(x))-a(\phi(y))
$$
where the scalars $\lambda_{x\to y\to x}$ are given by (\ref{rem-met}).
\end{proposition}

\begin{proof}
We notice that $\theta=\sum_i \theta_i$ where $\theta_i \in \Gamma^1_i$
in the decomposition $\Gamma^1_G=\sum \Gamma^1_i$ for the given \'etale directed cover $\phi$.
Then, from the very definition, $da=\sum_i [\theta_i, d_ia]$. The formula
follows from
a straightforward calculation
similar to the one for the case of $\diGraphs_{\leq 1}$ as in \cite{dimakis} and
\cite[Ch. 1]{bm}.
\end{proof}

We end this section by relating our treatment with the Laplacian for
standard graphs as in \cite{diestel,godsil}.

\begin{observation}
By Prop. \ref{prop-graphffemb} any graph in $\Graphs_{\leq 1}$ corresponds to
an element in $\diGraphs$, through a fully faithful functor. This correspondence
does not amount to give an orientation on the graph, but consists in giving two oriented edges
for each unoriented edge. Notice that, by Prop. \ref{met-expr}, such property is a necessary
condition for the definition of a generalized quantum metric to make sense.

In the classical literature \cite{diestel,godsil},
the Laplacian operator $L$ is defined for a standard
(unoriented) graph $G$ in $\Graphs_{\leq 1}$ as follows:
$$
L:\bk[V_G] \lra \bk[V_G], \qquad (La)(x)=\sum_{y, (x,y) \in E_G} (a(x)-a(y))
$$
If we fix the basis $\{\de_x\}_{x\in V_G}$ for $\bk[V_G]$, we identify $\bk[V_G]\cong \bk^{|V_G|}$,
hence $L$ is a linear operator and one can readily check:
$$
L=D-A
$$
where $D$ is the degree matrix (diagonal matrix with the degree of vertices on
the diagonal) and $A$ is the adjacency matrix of $G$.
Let 
$\Delta_\theta$ the graph Laplacian
for $\Theta(G)$ with respect to the metric given by $\lambda_{v\to w, w\to v}
=\lambda_{w\to v,v\to w}=1$. 
Then, we have:
$$
L=(1/2)\Delta_\theta
$$
This is immediate from the comparison of the expressions of $L$ and $\Delta_\theta$:
$$
(La)(x)=\sum_{y, (x,y) \in E_G} (a(x)-a(y)), \qquad \Delta_\theta a(x)=
2\sum_{y, x\to y \in E_G} \lambda_{x \to y,y\to x} (a(x)-a(y))
$$
\end{observation}

\subsection{Connection Laplacian}
We can extend the definition $_\theta\Delta$ when a right connection
is given.
\begin{definition}
Let $G\in\diGraphs_{\leq 1}$, $M$ a free rank $n$ right
$A_G$-module, $\nabla$ a right connection $\nabla$ and $(,)$ a generalized quantum metric on $\Gamma^1_G$.
Denote as $\eta$ the left $A_G$-module map $M\otimes_A\Gamma^1_G\rightarrow M\otimes_A\Gamma^1_G\otimes_A\Gamma^1_G$
defined on indecomposable tensors as follows
$$
\eta(m\otimes\oomega_{x\to y})=m\otimes\oomega_{x\to y}\otimes \theta
$$
    
Define the
\textit{connection Laplacian} $_{\theta}\Delta^M$ as $$_{\theta}\Delta^M:=-2(\id\otimes(,))\circ\eta\circ\nabla:M\rightarrow M$$
\end{definition}
\begin{observation}
If $M=A_G$, 
we recover the Laplacian as in Def. 
\ref{graph-lapl}.
\end{observation}

We now give an explicit expression for $_{\theta}\Delta^M$.

\begin{observation}
Let be $G\in\diGraphs_{\leq 1}$, $M$ a free rank $n$, 
$A_G$-bimodule with
basis $\{ e_i\}_{i=1}^n$, $\nabla$ a right connection and $(,)$ a
generalized quantum metric on $\Gamma^1_G$ determined by scalars $\lambda_{x\to y\to x}$. Then
$$
_\theta\Delta^M(e_if^i)=-2\sum_{x\to y}\lambda_{x\to y\to x}
(\mathcal{R}^j_{x\to y,i}f^i(y)-f^i(x))e_j\delta_x
$$
where in the summation only the edges $x\to y$ of the maximal bidirected subgraph of $G$ appear.
\end{observation}

We now give another algebraic description of a Laplacian operator and its relation
with the connection Laplacian.

\begin{definition}\label{def:bochnerlapl}
Consider a digraph $G$, a free finite rank $n$
$A_G$-bimodule $M$, a connection
$\nabla:M\rightarrow M\otimes\Gamma^1_G$. Let us fix
isomorphisms $M \cong M^*$ and $\Gamma^1\cong (\Gamma^1)^*$.
We define the \textit{Bochner Laplacian},
as $\Delta^{\mathrm{B}}:= \nabla^*\nabla:A\rightarrow A$,
where we have denoted as $\nabla^\ast$ the dual (as map between vector spaces) of $\nabla$.
\end{definition}

\begin{proposition}\label{prop:bochner}
Let  $G\in\diGraphs_{\leq 1}$ be a bidirected graph, $M$ a rank $n$ free $A_G$-bimodule $(A, \Gamma^1)$ and $(,)$ the generalized quantum metric determined by scalars $\{\lambda_{x\to y\to x}\}$ all equal to $1$. Consider $\nabla$, $\nabla^*$
a noncommutative right connection, fixing isomorphisms $M \cong M^*$, $\Gamma^1\cong (\Gamma^1)^*$.
If the weak parallel transport $\{\mathcal{R}_{x\to y}\}$ associated to $\nabla$
is a parallel transport with all orthogonal linear isomorphisms $\{\mathcal{R}_{x\to y}\}$,
then we have  $\Delta^{\mathrm{B}}=-_\theta\Delta^{M}$.
\end{proposition}

\begin{proof} This is a direct check.\end{proof}

Consider a bidirected graph $G\in\diGraphs_{\leq 1}$ and let $\cF\in\Pre(P_G^{op})$ be a vector bundle of rank $n$ or a sheaf of inner product spaces (or Hilbert spaces) for the site $(\open(A(P_G)),\cT_{A(P_G)})$. Recall that we defined in Definition \ref{def:partranspvect} the parallel transport associated to $\cF$ as 

$$
{\cR^\cF:=}\lbrace 
\cF_{u\leq u\to v}^{\ast}\cF_{v\leq u\to v}\rbrace_{u\to v\in E_G}
$$
where we have denoted as $\cF_{v\leq e}^{\ast}$, abusing notation, either the adjoint of $\cF_{v\leq e}$ or $\cF_{v\leq e}^{-1}$ in the case $\cF$ is a vector bundle.
Recall that we saw in \ref{GB-Majid-equiv} that this datum corresponds to a connection that, once a frame bundle is chosen, can be explicitly written, according to (\ref{conn-f}), as:
\begin{equation}\label{eq-bod}
\nabla: e_if^i \mapsto \sum_{x \to y}\left[ f^i(y) (\cF^*_{x\leq x\to y}\cF_{y\leq x\to y})_i^j e_j - f^i(x)e_i\right]\otimes \oomega_{x\to y}
\end{equation}
This observation allows us to prove the following result.

\begin{theorem}\label{theo:sheaf_lap}
Let $G\in \diGraphs_{\leq 1}$ be a bidirected graph, $\cF$ a vector bundle of rank $n$ on $G$, $M$ the free $A_G$-bimodule associated to it. Assume to have isomorphisms $M \cong M^*$ and $\Gamma^1\cong (\Gamma^1)^*$ as in Def. \ref{def:bochnerlapl} and consider $(,)$ the generalized quantum metric determined by scalars $\{\lambda_{x\to y\to x}\}$ all equal to $1$. Then:

\begin{enumerate}
\item If $\mathcal{R}^\cF$ is a parallel transport, then $_\theta\Delta^{M} = -\Delta_{\cF}$.
\item If $\cF$ is a sheaf of inner product spaces and $\cF_{v\leq e}^\ast = \cF_{v\leq e}^{-1}$
i.e. $\cF$ in an $\mathrm{O}(n)$-bundle,
then $\Delta^{\mathrm{B}} = L_{\cF}$.
\end{enumerate}
\end{theorem}
\begin{proof}
We start from the first assertion. Let us begin by rewriting (\ref{sheaf-lap})
for a vector $f\in C^0 (G, \cF)$:

\begin{equation}\label{sl-proof}
\begin{array}{rl}
L_\cF(f)_x &= \sum_{x\to v} \cF^{-1}_{x\leq x\to v}(\cF_{x\leq x\to v}f_x-\cF_{v\leq x\to v}f_v)\\
& +\sum_{u\to x} \cF^{-1}_{x\leq u\to x}(\cF_{x\leq u\to x}f_x-\cF_{u\leq u\to x}f_u) \\ \\
&= \sum_{x\to v} (f_x-\cF^{-1}_{x\leq x\to v}\cF_{v\leq x\to v}f_v) \\ 
& +\sum_{u\to x} (f_x-\cF^{-1}_{x\leq u\to x}\cF_{u\leq u\to x}f_u) \\
\end{array}
\end{equation}
where the second equality holds because $\cF^{-1}_{u\leq u\to v}\cF_{u\leq u\to v}=\mathrm{id}$ for all $u\to v$ thanks to our assumptions on $\cF$. Using equation \eqref{eq-bod} and the fact that $\mathcal{R}^\cF$ is a parallel transport one can get the result. The proof of the final part of the statement then follows from the previous proposition.
\end{proof}

\begin{remark}
Let $H\in \stGraphs$ an undirected graph and consider $G=\Theta(H)$ and $\cF=\Theta_p(\mathcal{L})$ for $\mathcal{L}$ vector bundle of rank $n$ on $H$. Then $\mathcal{R}^\cF$ is a parallel transport, thus the hypothesis of Theorem \ref{theo:sheaf_lap} is always satisfied for this special case.
\end{remark}

\begin{remark}
The definition of Bochner Laplacian and the
previous Theorem can be generalized to the case in which the identification
with the duals appearing in it comes from inner products/adjoints.
\end{remark}

\begin{observation}
One can link in an obvious way the Laplacians considered in
Theorem \ref{theo:sheaf_lap} with the ones arising on an unoriented graph using
Observation \ref{Obs:bodnarlapl}.
\end{observation}

\begin{remark}
    Much of the previous treatment can be repeated mutatis mutandis in the case of left connections. We limit ourselves to list the main differences from the case of right connections, leaving the details for the interested reader.
    \begin{itemize}
        \item We have a bijective correspondence between left connections and weak parallel transports, but in this case the latter definition has to be modified reversing the direction of the linear maps appearing in it for the discussion to make sense in applications.
        \item We can in this case define a Laplacian $\Delta_\theta^M$ for a given connection on a left (free of finite rank) module $M$.
        \item If we consider a bidirected digraph $G\in\diGraphs_{\leq 1}$, using the canonical $\ast$-DGA structure on $\Gamma^\bullet_G$ (the one satisfying $\oomega_{x\to y}^\ast=-\oomega_{y\to x}$, see \cite[Chapter 1]{bm}), given a free $A_G$-bimodule of rank $n$ and a right connection on it having associated $1$ forms $\{\oomega^j_i\}$ we can define a left connection characterized by the forms $\{\oomega^{j\ast}_i\}$. In this case $_\theta\Delta^M=\Delta^M_\theta$ (mimicking \cite[1.18]{bm}).
    \end{itemize}
\end{remark}

\end{document}